\documentclass[11pt]{amsart}  

\usepackage[T1]{fontenc}

\usepackage{tikz}
\usetikzlibrary{angles,quotes,calc,decorations.pathreplacing}

\usepackage{amsmath,amssymb,mathtools}
\usepackage{microtype}
\usepackage{enumitem}
\usepackage{hyperref}

\usepackage[margin=1.1in]{geometry}
\newtheorem{theorem}{Theorem}[section]
\newtheorem{proposition}[theorem]{Proposition}
\newtheorem{lemma}[theorem]{Lemma}
\newtheorem{claim}[theorem]{Claim}
\newtheorem{fact}[theorem]{Fact}
\theoremstyle{definition}
\newtheorem*{definition}{Definition}
\theoremstyle{remark}
\newtheorem*{remark}{Remark}

\newcommand{\R}{\mathbb R}

\newcommand{\Z}{\mathbb Z}
\newcommand{\ip}[2]{\langle #1,#2\rangle}
\newcommand{\cP}{\mathcal P}
\newcommand{\cQ}{\mathcal Q}
\newcommand{\cR}{\mathcal R}
\newcommand{\bad}{\mathcal B}
\newcommand{\cV}{\mathcal V}
\newcommand{\1}{\mathbf 1}
\newcommand{\ang}{\angle}
\newcommand{\dist}{\operatorname{dist}}
\newcommand{\conv}{\operatorname{conv}}
\newcommand{\per}{\operatorname{per}}
\newcommand{\diam}{\operatorname{diam}}
\newcommand{\HH}{\mathcal H^1}
\newcommand{\eps}{\varepsilon}

\newcommand{\norm}[1]{\lVert #1\rVert}

\title{Volume stability of plank covers}

\author[Bakaev]{Egor Bakaev}
\address{Department of Computer Science, University of Copenhagen}

\author[Yehudayoff]{Amir Yehudayoff}
\address{Department of Computer Science, University of Copenhagen,
and Department of Mathematics, Technion-IIT}

\dedicatory{Dedicated to Avi Wigderson in celebration of his 70th birthday.}

\begin{document}
\begin{abstract}We prove a longstanding conjecture of András Bezdek on the stability of plank covers of the planar disk \(D\). If a sufficiently small disk \(\eps D\) is removed, then every finite family of planks covering the resulting annulus can be rearranged to cover the whole disk. For the proof, we establish a stronger dimension-independent stability theorem. If $K \subset \R^n$ is an origin-symmetric convex body for $n \geq 2$, and the uncovered region lies in \(\eps K\) and has volume \(v\), then the  excess overlap is at least \(c v/\eps\), where \(c>0\) is an absolute constant. The proof introduces a pruning procedure and a flattening mechanism. Pruning peels away curvature through local removals whose cumulative effect is guaranteed to be substantial. Flattening trades straightness for bounded multiplicity
and geometrically exposes the overlap that pays for each pruning step. 
\end{abstract}
\maketitle
\section{Introduction}

A plank is the closed region between two parallel hyperplanes; its width is the distance between them.
The classical plank problem, traditionally attributed to Tarski~\cite{tarski1932uwagi}, asks whether covering a convex body with several planks can require less total width than covering it with a single plank.
Bang's celebrated theorem~\cite{bang1950covering,bang1951solution} shows that the answer is no: a single plank of minimum width is optimal.
The problem and the methods developed to solve it have connections to Diophantine approximation and the geometry of numbers~\cite{alexander1968problem,davenport1962note,kupavskii2017tarski},
geometric functional analysis~\cite{ball1991plank,ball2001complex,munoz2010real},
polynomial methods~\cite{glazyrin2023covering,glazyrin2024extensions,ortegaMoreno2024complex},
harmonic analysis~\cite{nazarov1997bang},
symplectic geometry~\cite{akopyan2019bang},
and hyperplane-cover problems in combinatorics and theoretical computer science~\cite{yehuda2024lower}.
For surveys, see~\cite{bezdek2013tarski,brass2005research,verreault2026plank}.

Andr\'as Bezdek asked about the stability of Bang's theorem~\cite{bezdek2003covering}. How much can the covering cost decrease if some small region is allowed to remain uncovered? Specifically, he asked for which planar convex bodies $K$, there exists $\varepsilon>0$ such that if $\varepsilon K\subset\operatorname{int}K$ then covering the annulus $K\setminus\varepsilon K$ requires the same cost as covering $K$. He proved the square-annulus theorem that states stability when $K$ is a square. 
The property is not universal and does not hold for every convex body: it fails, for example, for equilateral triangles.

The case that $K$ is the disk
$D=\{x\in\R^2:\norm{x}\leq 1\}$ and the hole is concentric remained the central open case of A.~Bezdek's conjecture; see~\cite{ambrus2025covering, bezdek2003covering, bezdek2013tarski, brass2005research, fejestoth2023four, kupavskii2026nondissective, smurov2010covering, verreault2026plank, white2007covering, zhang2008note}.
Progress on other cases included results for polygonal annuli, higher-dimensional cubes, and other families of convex bodies \cite{ambrus2025covering,white2007covering,zhang2008note,smurov2010covering}; see also \cite{bezdek2013tarski,brass2005research}.  
W.~Kuperberg solved the analogous annulus problem in $3$-dimensional space~\cite{bezdek2003covering}. 
The reason is that if the planks cover the unit sphere, a width-$w$ plank cuts out a spherical zone of area at most $2\pi w$, and the sphere has area $4\pi$.

Bezdek's disk problem has been open for more than twenty five years, and we resolve it.

\begin{theorem}\label{thm:annulus}
There is a constant $\eps >0$ such that if a finite family of planar planks $\cP$ covers the annulus $D\setminus (\eps D)$ then
\[
 \sum_{P \in \cP} w(P) \ge2.
\]
\end{theorem}

 Our solution identifies a stability refinement of the trivial covering inequality. If planks cover a planar set of area $a$, then the sum of their  areas in the set is at least $a$. What happens if a small central region is allowed to remain uncovered? Can the resulting hole reduce the covering mass? Our main result shows that, for the disk, a sufficiently small central hole forces the covering mass to exceed the area of the entire disk. Figure~\ref{fig:plank} illustrates this phenomenon: the two central holes force a substantial amount of overlap, as quantified by the following theorem.

\begin{remark}
Our proof of Theorem~\ref{thm:annulus} was formalized in Lean 
by Arseniy Akopyan \cite{palomar-2026-10-08-000004-v1}.
\end{remark}

\begin{figure}[]
    \centering \includegraphics[width=0.5\textwidth]{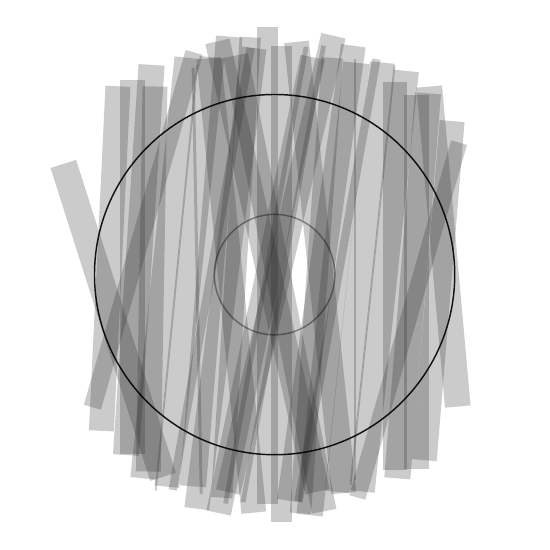}
    \caption{An illustration of planks covering the annulus.}
    \label{fig:plank}
\end{figure}

\begin{theorem}\label{thm:hole-cost}
There are constants $c_0,r_0>0$ such that the following holds.
Let $\cP$ be a finite family of planar planks  and denote the hole by
\[
 H := D \setminus \bigcup \cP .
\]
If $H \subset rD$ for $0 < r < r_0$ then
\begin{equation*}
\sum_{P \in \cP}|P \cap D|
\geq |D| +
\frac{c_0}{r}|H|,
\end{equation*}
where $|\cdot|$ denotes area.
\end{theorem}

\begin{remark}
Theorem~\ref{thm:annulus} immediately follows from Theorem~\ref{thm:hole-cost} using the classical relative-width measure on the disk \cite{gardner1988relative}. 
For completeness, this reduction is described in Section~\ref{sec:CostToB}.   
\end{remark}

Theorem~\ref{thm:hole-cost} shows that if planks cover a punctured disk then the planks' mass inside $D$ can be moved around to cover all of $D$ (including the hole). It also shows that the only configuration where there are no leftovers after moving mass around is when there is no hole and there is essentially a single plank.

There is an equivalent formulation using excess overlap. 
 Define the counting function
\[
 M_{\cP}(x)=\sum_{P\in\cP}\1_P(x)
\]
so that $\int_D M_\cP = \sum_P |P \cap D|$.
Define the overlap cost
\[
 I(\cP)
 =\int_D (M_{\cP}(x)-1)_+\,dx.
\]
The theorem says that
the ratio between the overlap $I(\cP)$ and the hole area $|H|$ tends to infinity as $r$ tends to zero; more precisely, 
$I(\cP) \gtrsim \frac{1}{r} |H|$.

Theorem~\ref{thm:hole-cost} is in fact a special case of a more general phenomenon. 
First, it applies to all origin-symmetric convex bodies $K$, not just the disk $D$. The reason is that the assertion is invariant under affine maps, and John's theorem says that 
$K \subset \R^2$ has a position $\widetilde K$ such that
$D \subseteq \widetilde K \subseteq \sqrt{2}D$.
Second, as observed by William Verreault, by considering $2$-dimensional sections it holds for the $n$-dimensional ball for all $n \geq 2$. 
Both observations lead to the following general dimension-independent statement
(the proof appears in Section~\ref{sec:high-dim}).

\begin{theorem}\label{thm:hole-cost-Rn}
There are constants $c_1,r_1>0$ such that the following holds.
Let $K \subset \R^n$ be an origin-symmetric convex body.
Let $\cP$ be a finite family of planks in $\R^n$ and denote the hole by
\[
 H := K \setminus \bigcup \cP .
\]
If $H \subset rK$ for $0 < r < r_1$ then
\begin{equation*}
\sum_{P \in\cP} |P \cap K|
\geq |K| + \frac{c_1}{r}|H| ,
\end{equation*}
where $|\cdot| = |\cdot|_n$ denotes the $n$-dimensional Lebesgue measure.
\end{theorem}

Going back to Tarski's original motivation, the theorem provides a lower bound on the number of pieces in certain dissections of symmetric convex bodies. Assume that we partition a symmetric convex body $K$ into $N$ measurable pieces $K_1,\ldots,K_N$ allowing shared boundaries of measure zero so that we can rigidly move the $K_i$'s into a single plank $P_*$ of width $w_*$. By pulling back each piece $K_i$, we get a plank $P_i$ of width $w_*$, and these planks cover $K$ so that
\[|K| \leq \sum_i |K \cap P_i| \leq
N g_K(w_*)\]
where $g_K(w_*)$ is the supremum of $|K \cap P|$ over all planks $P$ of width $w_*$. In other words, if $w_*$ is small then many pieces are needed. 
Theorem~\ref{thm:hole-cost-Rn} shows that the same bound $N \geq |K| / g_K(w_*)$ on the number of pieces holds when we are only required to partition the annulus $K \setminus (rK)$. 

The theorem admits an analytic interpretation. An integer-valued ridge function has the form
\[
F(x)=f(\ip{x}{u}),\qquad f:\R\to\Z_{\geq0},
\]
where $u\in\R^n$ is a unit vector. If $f$ is compactly supported with finitely many jumps, then $F$ is essentially a finite sum of indicators of parallel planks.
There is also a Fourier interpretation. The Fourier transform of a ridge function $F(x)=f(\ip{x}{u})$ is supported on the line $\R u$. A finite sum of plank indicators has Fourier support contained in finitely many lines through the origin.
The theorem can be viewed as a one-sided $L^1$ stability estimate for sums $M$ of such ridge functions. If the zero set of $M$ within the body is confined to a small central region, then the positive discrepancy $\int(M-1)_+$ must greatly exceed the negative discrepancy $\int(1-M)_+$.

It suffices to prove the result in the plane, so we consider the plane up until Section~\ref{sec:high-dim}. The guiding principle is that holes create curvature, whereas planks are straight. By the planar Gauss--Bonnet theorem, the boundary of each hole has total turning $2\pi$. Accommodating this turning with straight planks forces overlap. The smaller $r$ is, the greater the distance to the boundary of $D$ relative to the scale of the hole, and the larger the resulting overlap cost. Turning this geometric intuition into a quantitative estimate requires several new ideas. We begin by explaining some obstacles.

The planks can be thought of, without loss of generality, as many thin equal-width planks.
Any proof of the above statement must therefore overcome the surprising nature of line configurations in the plane. For example, the existence of
Besicovitch--Kakeya constructions that pack segments in many directions into arbitrarily small area \cite{besicovitch1928kakeya}, and Davies's construction 
showing how to cover a  planar body by lines with zero area outside the body
\cite{davies1952accessibility}.  
Davies’s theorem allows an annulus to be covered by lines with no additional area. Our theorem reveals an unavoidable cost such constructions must pay. Every finite-plank approximation of Davies's construction must pay in overlap.

Classical \(L^2\) bounds in Kakeya theory exploit the identity \(\int_D M_\cP^2=\sum_{P,Q}|P\cap Q \cap D|\), which reduces the analysis of multiplicity to pairwise intersections.
The \(L^1\) overlap cost \(\int_D (M-1)_+\) is harder to control from below,
because a large intersection energy may be concentrated on a tiny region and yield little \(L^1\) overlap. Proving an \(L^1\) lower bound therefore requires geometric information beyond pairwise intersection estimates.

The fact that lines are straight is actually a source of difficulty for us. 
Curved objects cannot concentrate in the same way as straight ones, because curvature constrains 
how many of them can  nearly overlap for long distances. Stein proved the spherical maximal theorem in dimension $\ge3$ \cite{stein1976maximal}. The harder planar case was settled by Bourgain in his circular maximal theorem~\cite{bourgain1986averages}. Schlag exposed the connections to incidence geometry of nearly tangent circles \cite{schlag1997generalization,schlag1998geometric,schlag2003continuum}. Wolff showed that a planar Borel set containing a circle of every radius has Hausdorff dimension two~\cite{wolff1997kakeya}.

There is also a genuine integrality phenomenon. 
There exists a fractional cover that pays nothing in overlap.
A fractional cover is the linear-programming relaxation of an ordinary cover. It is a non-negative measure on the set of planks such that every point is covered by plank mass at least \(1\). For example, there is a fractional plank 
cover with value $1$ on $D \setminus (D/2)$ and value $0$ on $D/2$.
All planks are tangent to $\partial(D/2)$ from the outside and are distributed as follows. Let $P_{\theta,h}=\{x:\ \tfrac12\le\langle x,e_\theta\rangle\le\tfrac12+h\}$ with $\theta$ uniform on $[0,2\pi)$ and the width $h\in(0,\tfrac12]$ distributed so that the weight of planks of width at least $h$ is some $\Gamma(h)$. 
For $x \in D \setminus (D/2)$, let
$
\alpha=\arccos(\frac1{2\|x\|} )$.
The fractional mass covering $x$ is
$
\frac1\pi\int_0^\alpha \Gamma\!\left(\|x\| \cos\theta-\frac12\right)\,d\theta$.
This is an Abel integral equation, and one can choose
$\Gamma$ so that it is $1$ for every such $x$. 

The existence of a lossless fractional cover places a constraint on possible proofs. The natural linear programming 
relaxation admits a lossless solution, so its dual cannot certify a positive overlap cost. Encoding the planks by convex potentials and applying complex analytic tools faces the same obstruction unless the argument retains information about the integer-valued multiplicity. Likewise, Fourier methods based only on linear identities or inequalities often cannot distinguish the two settings. Symmetrization can erase precisely this distinction, since averaging rotated covers produces a fractional cover even when the original cover is integral. A successful proof must exploit a feature of integral covers that is lost under fractional relaxation or averaging.

\subsection{Proof sketch}
\label{sec:proofSketch}

The proof has two main new ingredients: a pruning procedure and a flattening mechanism. The proof outline below explains their guiding ideas and how they work together. Definitions are informal here; precise statements, complete proofs, and accompanying figures appear in the subsequent sections.

\begin{remark}
By subdividing the planks, we may and will assume that they all have the same
very small width $w$.
This means that we think of the planks as being very close to straight lines.     
\end{remark}

\subsection*{\em The pruning procedure and growth of the hole (Section~\ref{sec:prun}).}
Fix $\rho>0$ small enough so that $\rho D$ lies well inside \(D\) and boundary effects can be ignored. A pair of planks $(P,Q)$ in $\cP$ is a candidate pair if their center lines cross inside~$\rho D$.
Associate with such a pair the convex hull
\[
K(P,Q)=\conv\big((P\cap \rho D)\cup(Q\cap\rho D)\big).
\]
The pair $(P,Q)$ is removable if a non-negligible fraction of $K(P,Q)$
is covered by planks in $\cP$ whose directions are close to the direction of
$P$ or $Q$. 
The pruning procedure is now straightforward:
\begin{quote}
Starting with the original family,  repeatedly delete
removable pairs until no removable pair remains.
\end{quote}
The main point is that the hole of the final family is much larger
than the original hole. If $H'$ is the final hole,
then
\[
|H'|\gtrsim \frac{\rho}{r}|H|.
\]
We now describe an explanation.
Decompose $H'$ into its convex cells. Among the cells meeting $rD$, call a cell bad if its total area is not much larger
than its area inside $rD$, and otherwise call the cell good.
A good cell shows that the new hole is larger than the original hole, so we would like to say that the total mass of bad cells is small.

Elementary geometry shows that every bad cell is confined to a small
neighborhood of $rD$. More precisely,
it has a boundary plank $P$ such that the entire cell lies in a thin region adjacent to $P$. 
In this thin region, we locate
a rectangular band $\Sigma_P$ that lies almost entirely in $H'$
(because there are no more removable pairs).
The dimensions of $\Sigma_P$ are $\rho \times h_P$  
where the height $h_P$ is proportional to the largest angle $\alpha_P$ at which $P$ meets another retained plank.

For each bad cell, we thus associate a rectangular band $\Sigma_P$ of height $h_P$. By construction, 
all of these bands are somewhat parallel, so a Vitali-type selection (a.k.a.\ greedy) allows us to choose disjoint bands $\{\Sigma_j\}$ so that after a fixed enlargement they cover all bad cells:
\[
\sum_{G\ {\rm bad}} |G\cap rD|
\lesssim r\sum_j h_j .
\]
Because the disjoint bands lie almost entirely in $H'$,
\[
\rho\sum_j h_j\lesssim |H'|.
\]
Together, 
\[
\sum_{G\ {\rm bad}} |G\cap rD|
\lesssim \frac r\rho |H'| .
\]
For the good cells, by definition,
\[
\sum_{G\ {\rm good}} |G\cap rD|
\lesssim \frac r\rho |H'| .
\]
Finally,
\[
|H|
\le
\sum_{G\ \mathrm{good}}|G\cap rD|
+
\sum_{G\ \mathrm{bad}}|G\cap rD|
\lesssim \frac r\rho|H'|.
\]

\subsection*{\em Flattening and why removable pairs force overlap (Section~\ref{sec:removablePair}).}
The main next claim is that every removable pair has a definite
overlap cost. Fix a removable pair $(P,Q)$ with angle $\beta = \ang(P,Q)$.
The convex region $K(P,Q)$ comprises  four triangle-like regions with the crossing point of $P,Q$ as the apex. 
Because $(P,Q)$ is removable, 
one of these triangles~$T$
has a packet
of planks that are nearly parallel to one of the arms and covers a
fixed fraction of~$T$.
The two arms of $T$ correspond to the two planks $P$ and $Q$.

A major difficulty is that this packet can have arbitrarily large multiplicity and highly irregular Kakeya-like structure. The flattening procedure we develop allows us to replace the plank-packet
by a bounded-multiplicity family of Lipschitz ribbons. We lose some geometric data by replacing straight lines by Lipschitz curves, but we win analytically because the multiplicity becomes bounded.
The Lipschitz control of the curves is weaker but it still allows us to deduce the key topological fact that every curve crosses at least one of the two arms of $T$. This crossing eventually leads to overlap.

\begin{remark}
The flattening mechanism preserves the Lipschitz constant and is therefore applicable not just to straight planks but also to any setting where Lipschitz curves are involved. It thus provides a general bridge between $L^1$-type overlap and geometry. Our first analysis of the mechanism was based on linear programming duality and planar geometry. However, as we describe next, we found a more direct algorithmic approach that is in fact more informative.  
\end{remark}

\begin{figure}[]
    \centering 
    \begin{tikzpicture}[scale=1.2,
  x=.72cm,y=2.20cm,
  line cap=round,line join=round,
  every node/.style={font=\scriptsize}
]

\tikzset{
  panel/.style={draw=black!35,line width=.45pt},
  rawline/.style={draw=black!23,densely dashed,line width=.42pt},
  fline/.style={draw=black,line width=.88pt}
}

\begin{scope}
  \clip (-4.0,-1.97) rectangle (4.0,0.55);
  \draw[fline] plot coordinates {(-4.000,-1.920) (-2.500,-1.470) (-1.625,-1.278) (-0.875,-1.173) (0.000,-1.120) (0.875,-1.173) (1.625,-1.278) (2.500,-1.470) (4.000,-1.920)};
  \draw[fline] plot coordinates {(-4.000,-1.800) (-2.500,-1.470) (-2.062,-1.339) (-1.625,-1.278) (-1.250,-1.195) (-0.875,-1.173) (-0.350,-1.099) (0.000,-1.120) (0.350,-1.099) (0.875,-1.173) (1.250,-1.195) (1.625,-1.278) (2.062,-1.339) (2.500,-1.470) (4.000,-1.800)};
  \draw[fline] plot coordinates {(-4.000,-1.610) (-2.062,-1.339) (-1.667,-1.220) (-1.250,-1.195) (-0.714,-1.077) (-0.350,-1.099) (0.000,-1.050) (0.350,-1.099) (0.714,-1.077) (1.250,-1.195) (1.667,-1.220) (2.062,-1.339) (4.000,-1.610)};
  \draw[fline] plot coordinates {(-4.000,-1.360) (-1.667,-1.220) (-1.111,-1.053) (-0.714,-1.077) (-0.361,-0.999) (0.000,-1.050) (0.361,-0.999) (0.714,-1.077) (1.111,-1.053) (1.667,-1.220) (4.000,-1.360)};
  \draw[fline] plot coordinates {(-4.000,-0.880) (-1.111,-1.053) (-0.750,-0.945) (-0.361,-0.999) (0.000,-0.920) (0.361,-0.999) (0.750,-0.945) (1.111,-1.053) (4.000,-0.880)};
  \draw[fline] plot coordinates {(-4.000,-0.490) (-0.750,-0.945) (-0.385,-0.835) (0.000,-0.920) (0.385,-0.835) (0.750,-0.945) (4.000,-0.490)};
  \draw[fline] plot coordinates {(-4.000,-0.040) (-0.385,-0.835) (0.000,-0.720) (0.385,-0.835) (4.000,-0.040)};
  \draw[fline] plot coordinates {(-4.000,0.480) (0.000,-0.720) (4.000,0.480)};
\end{scope}
\draw[panel] (-4.0,-1.97) rectangle (4.0,0.55);

\begin{scope}[xshift=7.00cm]
  \clip (-4.0,-1.97) rectangle (4.0,0.55);
  \draw[draw=black!25,densely dashed,line width=.48pt] plot coordinates {(-4.000,-1.920) (-2.500,-1.470) (-1.625,-1.278) (-0.875,-1.173) (0.000,-1.120) (0.875,-1.173) (1.625,-1.278) (2.500,-1.470) (4.000,-1.920)};
  \draw[draw=black!25,densely dashed,line width=.48pt] plot coordinates {(-4.000,-1.800) (-2.500,-1.470) (-2.062,-1.339) (-1.625,-1.278) (-1.250,-1.195) (-0.875,-1.173) (-0.350,-1.099) (0.000,-1.120) (0.350,-1.099) (0.875,-1.173) (1.250,-1.195) (1.625,-1.278) (2.062,-1.339) (2.500,-1.470) (4.000,-1.800)};
  \draw[draw=black!25,densely dashed,line width=.48pt] plot coordinates {(-4.000,-1.610) (-2.062,-1.339) (-1.667,-1.220) (-1.250,-1.195) (-0.714,-1.077) (-0.350,-1.099) (0.000,-1.050) (0.350,-1.099) (0.714,-1.077) (1.250,-1.195) (1.667,-1.220) (2.062,-1.339) (4.000,-1.610)};
  \draw[draw=black!25,densely dashed,line width=.48pt] plot coordinates {(-4.000,-1.360) (-1.667,-1.220) (-1.111,-1.053) (-0.714,-1.077) (-0.361,-0.999) (0.000,-1.050) (0.361,-0.999) (0.714,-1.077) (1.111,-1.053) (1.667,-1.220) (4.000,-1.360)};
  \draw[draw=black!25,densely dashed,line width=.48pt] plot coordinates {(-4.000,-0.880) (-1.111,-1.053) (-0.750,-0.945) (-0.361,-0.999) (0.000,-0.920) (0.361,-0.999) (0.750,-0.945) (1.111,-1.053) (4.000,-0.880)};
  \draw[draw=black!25,densely dashed,line width=.48pt] plot coordinates {(-4.000,-0.490) (-0.750,-0.945) (-0.385,-0.835) (0.000,-0.920) (0.385,-0.835) (0.750,-0.945) (4.000,-0.490)};
  \draw[draw=black!25,densely dashed,line width=.48pt] plot coordinates {(-4.000,-0.040) (-0.385,-0.835) (0.000,-0.720) (0.385,-0.835) (4.000,-0.040)};
  \draw[draw=black!25,densely dashed,line width=.48pt] plot coordinates {(-4.000,0.480) (0.000,-0.720) (4.000,0.480)};

  \draw[draw=black,line width=.95pt] plot coordinates {(-4.000,-1.920) (-2.500,-1.470) (-1.625,-1.278) (-0.875,-1.173) (0.000,-1.120) (0.875,-1.173) (1.625,-1.278) (2.500,-1.470) (4.000,-1.920)};
  \draw[draw=black,dotted,line width=.95pt] plot coordinates {(-4.000,-1.800) (-2.500,-1.470) (-2.062,-1.339) (-1.625,-1.278) (-1.250,-1.195) (-0.875,-1.173) (-0.350,-1.099) (0.000,-1.120) (0.350,-1.099) (0.875,-1.173) (1.250,-1.195) (1.625,-1.278) (2.062,-1.339) (2.500,-1.470) (4.000,-1.800)};
  \draw[draw=black,line width=.95pt] plot coordinates {(-4.000,-1.610) (-3.188,-1.496) (-2.500,-1.290) (-1.625,-1.098) (-0.875,-0.993) (0.000,-0.940) (0.875,-0.993) (1.625,-1.098) (2.500,-1.290) (3.188,-1.496) (4.000,-1.610)};
  \draw[draw=black,dotted,line width=.95pt] plot coordinates {(-4.000,-1.360) (-2.417,-1.265) (-2.062,-1.159) (-1.625,-1.098) (-1.250,-1.015) (-0.875,-0.993) (-0.350,-0.919) (0.000,-0.940) (0.350,-0.919) (0.875,-0.993) (1.250,-1.015) (1.625,-1.098) (2.062,-1.159) (2.417,-1.265) (4.000,-1.360)};
  \draw[draw=black,line width=.95pt] plot coordinates {(-4.000,-0.880) (-2.000,-1.000) (-1.625,-0.918) (-0.875,-0.813) (0.000,-0.760) (0.875,-0.813) (1.625,-0.918) (2.000,-1.000) (4.000,-0.880)};
  \draw[draw=black,dotted,line width=.95pt] plot coordinates {(-4.000,-0.490) (-1.361,-0.859) (-1.250,-0.835) (-0.875,-0.813) (-0.350,-0.739) (0.000,-0.760) (0.350,-0.739) (0.875,-0.813) (1.250,-0.835) (1.361,-0.859) (4.000,-0.490)};
  \draw[draw=black,line width=.95pt] plot coordinates {(-4.000,-0.040) (-1.139,-0.669) (-0.875,-0.633) (0.000,-0.580) (0.875,-0.633) (1.139,-0.669) (4.000,-0.040)};
  \draw[draw=black,dotted,line width=.95pt] plot coordinates {(-4.000,0.480) (-0.477,-0.577) (-0.350,-0.559) (0.000,-0.580) (0.350,-0.559) (0.477,-0.577) (4.000,0.480)};
\end{scope}
\begin{scope}[xshift=7.00cm]
  \draw[panel] (-4.0,-1.97) rectangle (4.0,0.55);
  \node[anchor=west,fill=white,inner sep=.45pt,text=black,font=\tiny] at (-3.96,-1.884) {$g_1$};
  \node[anchor=west,fill=white,inner sep=.45pt,text=black,font=\tiny] at (-3.96,-1.774) {$g_2$};
  \node[anchor=west,fill=white,inner sep=.45pt,text=black,font=\tiny] at (-3.96,-1.593) {$g_3$};
  \node[anchor=west,fill=white,inner sep=.45pt,text=black,font=\tiny] at (-3.96,-1.353) {$g_4$};
  \node[anchor=west,fill=white,inner sep=.45pt,text=black,font=\tiny] at (-3.96,-0.887) {$g_5$};
  \node[anchor=west,fill=white,inner sep=.45pt,text=black,font=\tiny] at (-3.96,-0.507) {$g_6$};
  \node[anchor=west,fill=white,inner sep=.45pt,text=black,font=\tiny] at (-3.96,-0.066) {$g_7$};
  \node[anchor=west,fill=white,inner sep=.45pt,text=black,font=\tiny] at (-3.96,0.444) {$g_8$};
\end{scope}

\end{tikzpicture}
\caption{An illustration of the ``flatten the packet'' step. Starting from the midlines of the planks (left), we construct the functions $g_1,g_2,\ldots$ (right). The odd-indexed functions are drawn with solid lines, while the even-indexed functions are drawn with dotted lines. The functions satisfy the property $g_{i+2}\geq g_i+s$.}
    \label{fig:lipschitz}
\end{figure}

Flattening proceeds as follows; see Figure~\ref{fig:lipschitz} for an illustration. 
Choose a coordinate system where all relevant planks are Lipschitz. 
For a fixed $x$ value, order the
center-line points of the planks as
\[
f_1(x)\le f_2(x) \le f_3(x) \le \ldots .
\]
The functions $f_1,f_2,\ldots$ are no longer affine but they are still Lipschitz and are also piecewise affine and non-crossing.
Flatten the $f_i$'s inductively by defining $g_1=f_1$, $g_2=f_2$ and
\[
g_i=\max\{f_i,g_{i-2}+s\},
\]
where $s \sim w$ is some fixed shift size. 
In words, we keep pushing the $f_i$'s upwards as long as there is large local multiplicity.

Around each $g_i$ place a ribbon of vertical width comparable to $w$.
These ribbons still cover the original planks and retain the same Lipschitz behavior. The main advantage is that the ribbon multiplicity is bounded from above by an absolute constant, because
$g_{i+2}\ge g_i+s$.
We have thus replaced a collection of straight planks but with arbitrary multiplicity by a collection of arbitrary Lipschitz ribbons but with bounded multiplicity. 
The topological crossing property described above implies that the arms of $T$ are non-trivially covered by the ribbons.

Another central property is that ribbon-coverage can be explained by overlap. There are, in fact, two types of ribbon-coverage. The first is ribbon-coverage that is also plank-coverage, which is directly tied to overlap of planks.
The second is new ribbon-coverage.
The point is that there is new coverage only if mass was moved during flattening, and this too can be directly tied to overlap. Intuitively, a lot of mass being moved means that many of the $f_i$'s were concentrated in a small area, so there must be overlap.

More formally, a disk $B \subset D$ is called dense if
\[
\int_B (M_\cP(x)-1)_+ \; dx >\delta |B|
\]
for an absolute $\delta>0$.
Denote by $U$ the union of all dense disks. Another standard Vitali-cover argument relates the dense region to overlap: $\delta |U| \lesssim I(\cP):=\int_D (M-1)_+$.
The argument outlined above shows that the removable pair $(P,Q)$ pays dense area overlap:
\[
|U \cap(P\cup Q)\cap 5\rho D|
\gtrsim 
|(P\cup Q)\cap 5\rho D| \sim \rho w.
\]

\subsection*{\em From dense geometry to the final overlap bound.}
Because the deletion of one removable pair can enlarge the hole's area by at most $\sim w$, while the final hole has grown by a factor of $\sim \rho/r$, the number of pairs deleted by the pruning procedure is
\[
m\gtrsim \frac{\rho |H|}{rw}.
\]
We can deduce that
\[
\sum_i |S_i|
\gtrsim m\rho w
\gtrsim \frac{\rho^2}{r}|H| ,
\]
where $(P_i,Q_i)$ is the $i$'th deleted pair and
\[
S_i:=U \cap(P_i\cup Q_i)\cap5\rho D.
\]
On the other hand, because $\sum_i \1_{S_i}\le M_\cP$ and $\delta |U|\lesssim I(\mathcal P)$,
\[
\sum_i |S_i|\le |U|+I(\mathcal P)
\lesssim \frac{1}{\delta} I(\cP) .
\]
We can finally deduce
\[ I(\cP) \gtrsim \frac{\delta \rho^2}{r} |H|  \]
(recall that $\rho,\delta$ are some constants).

\subsection{AI methodology}

All major new definitions and constructions---including the pruning procedure, the flattening mechanism and the notion of dense disks---were developed without the support of any AI tool. Attempts to use these tools to come up with proofs were unsuccessful, with one exception: the Vitali-type selection in the analysis of the pruning procedure was worked out by Claude Fable 5.1.
The text itself was written with the help of both ChatGPT and Claude. 

\subsection{Acknowledgements} We wish to thank Zachary Chase, Fedor Nazarov and William Verreault for helpful remarks at various stages of this work. 
We also wish to thank Arseniy Akopyan for formalizing our proof in Lean.
The authors are supported by the DNRF Chair grant 20231101-28697.

\section{Pruning amplifies the hole}
\label{sec:prun}

In this section, we define and analyze a pruning procedure that repeatedly removes certain pairs of crossing planks. The main claim is that when the procedure terminates, the new uncovered set has expanded area-wise by a large factor compared to the original one.
Let $\cP$ be a finite family of planks, and denote the hole by $H:=D\setminus\bigcup P$. 
Figure~\ref{fig:candidate}
may help to grasp the following definitions. 

\begin{definition}
For a plank $P$, its distance from the origin is $d_P=\dist(0,\ell_P)$
where $\ell_P$ is the center line of $P$. Call $P$ \emph{deep} if $d_P\le0.75\rho$ and \emph{central} if $d_P\le0.07\rho$ where here and for the rest of the paper
\[\rho=\frac1{12}.\]

\end{definition}
\begin{definition}
A pair $(P,Q)$ of members of $\cP$ is a \emph{candidate pair} if $P$ is central, $Q$ is deep, and $\ell_P\cap\ell_Q$ is a single point in $\rho D$.  For a candidate pair, the set of almost parallel planks is denoted by 
\[
 N(P,Q):=\{R\in\cP:\min\{\ang(R,P),\ang(R,Q)\}<2\ang(P,Q)\} .
\]
Their convex hull region is
\[
 K(P,Q):=\conv\bigl((P\cap\rho D)\cup(Q\cap\rho D)\bigr) .
\]
The pair is \emph{removable} if
\begin{equation}\label{eq:removable}
 \left|K(P,Q)\cap\bigcup N(P,Q) \right|\ge \tau |K(P,Q)|,
\end{equation}
where here and for the rest of the paper
\[ \tau=10^{-4}.\]
\end{definition}

The pruning procedure is:
\begin{quote}
Starting from $\cQ=\cP$, as long as there is a removable candidate pair in $\cQ$ delete the pair from $\cQ$.     
\end{quote}

\begin{remark}
Different pruning orders may lead to different final families. Our results hold regardless of the pruning order. 
\end{remark}

\begin{definition}
A family without any removable candidate pairs is called \emph{final}.
\end{definition}

\begin{figure}[]
    \centering 
    \begin{tikzpicture}[scale=0.8,x=1cm,y=1cm,line cap=round,line join=round]
  \def\R{4.0}
  \def\rhoR{2.0}
  \def\halfplank{0.17}
  \pgfmathsetmacro{\centralR}{0.07*\rhoR}
  \pgfmathsetmacro{\deepR}{0.75*\rhoR}

  \newcommand{\Nplank}[3]{%
    \begin{scope}[rotate=#1]
      \path[fill=black,fill opacity=#3,draw=none]
        (-5.35,{#2-\halfplank}) rectangle (5.35,{#2+\halfplank});
    \end{scope}%
  }

  \Nplank{86}{-2.55}{0.19}
  \Nplank{83}{-1.85}{0.19}
  \Nplank{80}{-1.15}{0.19}
  \Nplank{76}{-0.55}{0.19}
  \Nplank{72}{ 0.78}{0.19}
  \Nplank{68}{ 1.48}{0.19}
  \Nplank{63}{ 2.10}{0.19}
  \Nplank{58}{-2.35}{0.19}
  \Nplank{54}{-1.60}{0.19}
  \Nplank{50}{-0.82}{0.19}
  \Nplank{47}{ 0.80}{0.19}
  \Nplank{40}{ 1.58}{0.19}
  \Nplank{36}{ 2.30}{0.19}

  \Nplank{78}{0.00}{0.48}
  \Nplank{43}{0.40}{0.48}

  \draw[black,opacity=0.72,line width=0.80pt]
    (-0.25,-1.98) -- (-0.58,-1.91) -- (-1.61,-1.19) -- (-1.79,-0.89)
    -- (0.25,1.98) -- (0.58,1.91) -- (1.01,1.72) -- (1.30,1.52) -- cycle;

  \draw[black,line width=0.72pt] (0,0) circle (\R);
  \draw[black,dashed,dash pattern=on 3pt off 2pt,line width=0.60pt]
    (0,0) circle (\rhoR);

  \node[font=\small,fill=white,fill opacity=.82,text opacity=1,inner sep=1.2pt]
    at (1.25,-0.88) {$K(P,Q)$};
  \node[font=\small,fill=white,fill opacity=.82,text opacity=1,inner sep=1.2pt]
    at (-2.15,1.35) {$\rho D$};
  \node[font=\small,fill=white,fill opacity=.82,text opacity=1,inner sep=1.2pt]
  at (-3.35,2.75) {$D$};
  \node[font=\small,fill=white,fill opacity=.82,text opacity=1,inner sep=1pt]
    at (0.25,3.20) {$P$};
  \node[font=\small,fill=white,fill opacity=.82,text opacity=1,inner sep=1pt]
    at (2.59,2.30) {$Q$};
\end{tikzpicture}
    \caption{An illustration of a candidate pair of planks $P,Q$. The central plank $P$ and the deep plank $Q$ are shown in dark gray, while planks from $N(P,Q)$ are shown in light gray. The outer circle is $D$, and the dashed inner circle is $\rho D$; for illustration, we take $\rho=0.5$. The boundary of $K(P,Q)$ is shown as well.}
    \label{fig:candidate}
\end{figure}

\begin{remark}
The set $N(P,Q)$ is always taken in the original family $\cP$. In particular, the property
``a candidate pair $(P,Q)$ is removable''
depends only on $P,Q$ and $\cP$ and not
on the position in the pruning order.  
\end{remark}

We may and will assume that all planks have common width $w$, and assume
\begin{equation}\label{eq:pruning-assumptions}
 w\le \frac r{100},\qquad 0<r\le\frac{\rho}{400},\qquad
 \varnothing\ne H \subset rD.
\end{equation}
Denote by $\ang(P,Q)\in(0,\pi/2]$ the acute angle between the center lines of the planks $P,Q$
(for parallel planks,
it is zero). We also assume that any two nonparallel members $P,Q\in\cP$ satisfy
\begin{equation}\label{eq:angle-thin-basic}
 w\le \frac{\rho}{50}\ang(P,Q).
\end{equation}
This assumption can be forced by subdivision of planks (which preserves angles) and a standard approximation argument. 

\begin{theorem}[new hole is large]\label{thm:pruning}
Assume \eqref{eq:pruning-assumptions} and \eqref{eq:angle-thin-basic}.
If $\cQ$ is final and its hole is
\[
H'=int(D) \setminus\bigcup \cQ,
\]
then
\begin{equation}\label{eq:pruning-amplification}
 |H'|\ge \frac{\rho}{200r}|H|.
\end{equation}
\end{theorem}

\begin{remark}
The reason we take the interior of $D$ is to get an open set $H'$.
\end{remark}
The rest of this section is devoted to proving the theorem.

\subsection{Elementary geometry}

Here we describe two simple properties of convex planar bodies and polygons. 
For $x\in\R^2$ and $s>0$, let
$B(x,s)=\{y\in\R^2:\norm{y-x}\le s\}$
denote the closed disk of radius $s$ centered at $x$.

\begin{lemma}[adding a far point to a convex body increases area]
\label{lem:convex-growth}
Let $G\subset\R^2$ be convex, let $K=G\cap rD\ne\varnothing$, and let $v\in G$ with $\norm v>r$.  Then,
\begin{equation}\label{eq:convex-growth}
 |G|\ge \frac{(\norm v-r)^2}{4r \norm v}|K|.
\end{equation}
\end{lemma}
\begin{proof}
Let $C$ be the union of all rays from $v$ through points of $K$. Since $K$ is convex, $C$ is a convex cone; let $\theta$ be its opening angle; see Figure \ref{fig:convex-growth}.
Since $K\subset rD$, $K$ lies in the annular sector centered at $v$:
\[
K\subset C\cap\{x:\norm v-r\le \norm{x-v}\le \norm v+r\},
\]
so
\[
|K|
\le \frac{\theta}{2}\bigl((\norm v+r)^2-(\norm v-r)^2\bigr)
=2r\norm v\,\theta.
\]

On the other hand, every ray in $C$ meets $K$ at distance at least $\norm v-r$ from $v$. Since $G$ is convex,
\[
C\cap B(v,\norm v-r)\subset G.
\]
Therefore
\[
|G|\ge \frac{\theta}{2}(\norm v-r)^2.
\]
Combining the two bounds gives \eqref{eq:convex-growth}.
\end{proof}

\begin{figure}[]
    \centering 
    \pgfmathdeclarefunction{HH}{1}{%
  \pgfmathparse{3.45 + 0.42*cos(2*#1) + 0.10*sin(3*#1)
                  + 0.04*cos(4*#1+25.7831) + 0.01*sin(5*#1-40.1070)}%
}
\pgfmathdeclarefunction{HP}{1}{%
  \pgfmathparse{-0.84*sin(2*#1) + 0.30*cos(3*#1)
                  - 0.16*sin(4*#1+25.7831) + 0.05*cos(5*#1-40.1070)}%
}
\pgfmathdeclarefunction{GX}{1}{%
  \pgfmathparse{4.25 + HH(#1)*cos(#1) - HP(#1)*sin(#1)}%
}
\pgfmathdeclarefunction{GY}{1}{%
  \pgfmathparse{0.00 + HH(#1)*sin(#1) + HP(#1)*cos(#1)}%
}

\begin{tikzpicture}[
  scale=0.8,
  x=1cm,y=1cm,
  line cap=round,
  line join=round,
  every node/.style={font=\small}
]
  \def\R{2.20}
  \coordinate (O) at (0,0);
  \coordinate (V) at (5.55,-0.65);

  \coordinate (A) at (1.196213,1.846368);
  \coordinate (B) at (1.306646,-1.769937);

  \def\angA{150.170967}
  \def\angB{194.784782}
  \def\dminus{3.387933}
  \def\dplus{7.787933}

  \begin{scope}
    \clip (O) circle (\R);
    \fill[gray!35]
      plot[domain=0:360,samples=241,smooth,variable=\t]
        ({GX(\t)},{GY(\t)}) -- cycle;
  \end{scope}

  \draw[black,dashed,line width=.72pt]
    ($(V)+(\angA:\dminus)$)
    arc[start angle=\angA,end angle=\angB,radius=\dminus];
  \draw[black,dashed,line width=.72pt]
    ($(V)+(\angA:\dplus)$)
    arc[start angle=\angA,end angle=\angB,radius=\dplus];

  \draw[black,dashed,line width=1.0pt] (V) -- ($(V)!1.75!(A)$);
  \draw[black,dashed,line width=1.0pt] (V) -- ($(V)!1.95!(B)$);

  \draw[black,line width=1.25pt]
    plot[domain=0:360,samples=241,smooth,variable=\t]
      ({GX(\t)},{GY(\t)}) -- cycle;
  \draw[gray!72,line width=1.20pt] (O) circle (\R);

  \pic[
    draw=black,
    line width=.82pt,
    angle radius=.72cm,
    angle eccentricity=1.34,
    "$C$"
  ] {angle=A--V--B};

  \fill (O) circle (1.55pt);
  \fill (V) circle (1.85pt);
  \fill (A) circle (1.35pt);
  \fill (B) circle (1.35pt);

  \node[below left=1pt] at (O) {$0$};
  \node[right=1pt] at (V) {$v$};
  \node[gray!75] at (-0.75,1.35) {$rD$};
  \node[right] at (7.05,-.62) {$G$};
  \node at (1.2,0.08) {$K$};

  \node[fill=white,inner sep=1.2pt] at (2.80,-1.90)
    {$\lVert v\rVert-r$};
  \node[fill=white,inner sep=1.2pt] at (-1.28,-2.99)
    {$\lVert v\rVert+r$};
\end{tikzpicture}
    \caption{An illustration for the proof of Lemma~\ref{lem:convex-growth}. The shaded set is
$K=G\cap rD$. The rays from $v$ are the boundary rays of the cone $C$,
and the dashed circular arcs centered at $v$ have radii $\|v\|-r$ and
$\|v\|+r$.}
    \label{fig:convex-growth}
\end{figure}

\subsection{The hull of a candidate pair}
\label{sec:hull-candidate-pair}
Here we describe a couple of properties of the convex hull of a candidate pair $(P,Q)$. 
Let $P$ be central and $Q$ deep with $\ell_P\cap\ell_Q=\{c\}\subset\rho D$, and set
\[
 \beta=\ang(P,Q)>0.
\]
Choose coordinates $(u,v)$ for the given plank $P$; see Figure~\ref{fig:candidate-pair-hull}. 
The line $\ell_P$ is the $u$-axis, and the coordinate origin is the perpendicular foot of the origin $0$ onto $\ell_P$. 
Thus, in these coordinates, the origin $0$ has coordinates $(0,\pm d_P)$.

Then,
\[
 \ell_P\cap\rho D=[-a_0,a_0]\times\{0\},
 \qquad a_0=\sqrt{\rho^2-d_P^2}\in[0.997\rho,\rho].
\]
Since $Q$ is deep, one of the two parts of
$\ell_Q\cap\rho D$ cut by $c$ has length at least
\[
\frac12|\ell_Q\cap\rho D|
=\sqrt{\rho^2-d_Q^2}
\ge \sqrt{1-0.75^2}\,\rho>0.66\rho.
\]
After reflecting the coordinate system across the $u$- and $v$-axes,
if necessary, we may assume that $\ell_Q$ makes angle $\beta$
counterclockwise from the positive $u$-axis and that the part of length at least $0.66\rho$ lies in the direction $\{v>0\}$.

\begin{figure}[]
    \centering 
\begin{tikzpicture}[
  x=1.25cm,y=1.25cm,
  line cap=round,line join=round,
  every node/.style={font=\small},
  point/.style={circle,fill=black,inner sep=1.45pt},
  axis/.style={->,semithick},
  main/.style={semithick},
]

\def\rad{3}
\def\dP{0.50}
\def\aZero{2.958040}
\def\cX{-1.2}
\def\betaAng{55}
\def\uE{1.103465}
\def\vE{3.289689}
\def\uF{-2.244043}
\def\vF{-1.491048}

\coordinate (F0) at (0,0);              
\coordinate (O)  at (0,\dP);            
\coordinate (L)  at (-\aZero,0);
\coordinate (R)  at ( \aZero,0);
\coordinate (C)  at (\cX,0);
\coordinate (E)  at (\uE,\vE);
\coordinate (Qm) at (\uF,\vF);
\coordinate (U)  at ($(C)+(1,0)$);

\fill[black!15] (L) -- (R) -- (E) -- cycle;
\draw[black!35,dashed] (L) -- (E) -- (R);
\node[black!55] at (1.62,1.28) {$T$};

\draw[main] (O) circle (\rad);
\node[anchor=west,font=\normalsize] at (2.65,2.28) {$\rho D$};

\draw[axis] (-3.9,0) -- (3.9,0) node[right] {$u$};
\draw[axis] (0,-2.75) -- (0,3.95) node[above] {$v$};
\node[anchor=east] at (-3.30,0.28) {$\ell_P$};

\draw[main] ($(Qm)!1.15!(E)$) -- ($(E)!1.15!(Qm)$);
\node[anchor=south east] at (-2.45,-1.85) {$\ell_Q$};

\node[point] at (L) {};
\node[point] at (R) {};
\node[point] at (C) {};
\node[point] at (F0) {};
\node[point] at (O) {};
\node[point] at (E) {};

\node[below left=2pt] at (L) {$(-a_0,0)$};
\node[below right=2pt] at (R) {$(a_0,0)$};
\node[below right=2pt] at (C) {$c$};
\node[below right=4pt] at (F0) {$(0,0)$};
\node[above right=3pt] at (O) {$0$};
\node[right=5pt, yshift=3pt] at (E) {$E=(u_E,v_E)$};

\pic[draw,angle radius=7mm,angle eccentricity=1.30,"$\beta$"]
  {angle=U--C--E};

\draw[decorate,decoration={brace,amplitude=5.5pt,raise=4pt},semithick]
  (C) -- (E)
  node[midway,sloped,above=10pt,fill=white,inner sep=1pt] {$\ge 0.66\rho$};

\draw[densely dotted] (O) -- (F0);

\end{tikzpicture}
    \caption{An illustration for the hull of a candidate pair; see Section~\ref{sec:hull-candidate-pair}. The lines $\ell_P$ and $\ell_Q$ intersect at $c$ inside the disk $\rho D$, and $E=(u_E,v_E)$ is the upper intersection of $\ell_Q$ with $\rho D$. The shaded triangle $T$ is contained in the hull $K(P,Q)$.}
    \label{fig:candidate-pair-hull}
\end{figure}

\begin{lemma}[the convex hull is not large]\label{lem:hull-upper}
\[
|K(P,Q)|\le2\rho^2\beta+10\rho w.
\]
\end{lemma}

\begin{proof}
Let
\[
K_0=\operatorname{conv}\bigl((\ell_P\cap \rho D)\cup(\ell_Q\cap \rho D)\bigr).
\]

Now we prove that
\begin{equation}\label{eq:neighborhood1.1}
    K:=K(P,Q)\subset K_0+B(0,1.1w).
\end{equation}

Take $R\in\{P,Q\}$ and choose coordinates in which
$
\ell_R=\{(u,d_R):u\in\mathbb R\}.
$
Since $R$ is deep,
\[
d_R\le0.75\rho,
\qquad
a:=\sqrt{\rho^2-d_R^2}\ge0.66\rho,
\]
and
$
\ell_R\cap\rho D=[-a,a]\times\{d_R\}.
$
Let $x=(u,v)\in R\cap\rho D$. If $|u|\le a$, then
$\dist(x,\ell_R\cap\rho D)\le w/2$. Otherwise, write
$|u|=a+s,$ $v=d_R-h,$
where $s>0$ and $0\le h\le w/2$. Since $x\in\rho D$,
\[
(a+s)^2+(d_R-h)^2\le a^2+d_R^2,
\]
and hence
$
2as\le2d_Rh\le d_Rw.
$
Thus
\[
s\le\frac{d_Rw}{2a}
\le\frac{0.75}{2\cdot0.66}w<0.57w,
\]
so
\[
\dist(x,\ell_R\cap\rho D)
\le s+h<1.1w.
\]
Therefore
\[
R\cap\rho D\subset
(\ell_R\cap\rho D)+B(0,1.1w),
\qquad R\in\{P,Q\},
\]
and taking convex hulls proves \eqref{eq:neighborhood1.1}.

The two chords have lengths at most \(2\rho\) and meet at angle \(\beta\), so
\[
|K_0|\le 2\rho^2\sin\beta\le 2\rho^2\beta,
\qquad
\operatorname{per}(K_0)\le 8\rho.
\]
By Steiner's formula and using \eqref{eq:pruning-assumptions},
\[
|K(P,Q)|
\le |K_0|+1.1w\,\operatorname{per}(K_0)+\pi(1.1w)^2
\le 2 \rho^2\beta+10\rho w.
\qedhere
\]
\end{proof}

\begin{lemma}[the convex hull contains a rectangle]\label{lem:hull-lower}
In the above coordinates,
\[
 K(P,Q) \supset
 \left\{(u,v):|u|\le0.6\rho,\ 0\le v\le0.1254 \rho\sin\beta \right\}.
\]
\end{lemma}

\begin{proof}
Let \(E=(u_E,v_E)\) be the endpoint of \(\ell_Q\cap \rho D\) with $v_E>0$;
see Figure \ref{fig:candidate-pair-hull}.
By the choice of coordinates, 
$v_E\ge 0.66\rho\sin\beta$ and
$|u_E|\le \rho$.
The triangle $T$ with vertices \((-a_0,0),(a_0,0)\), and \(E\) is contained
in \(K\). 

For $0\le t\le 1$, the horizontal section of $T$ at height $tv_E$
is the interval
\[
\left[
-(1-t)a_0+tu_E,\,
(1-t)a_0+tu_E
\right]\times\{tv_E\},
\]
which contains
$[-0.6\rho,0.6\rho]\times\{tv_E\}$
whenever
$
(1-t)a_0-t|u_E|\ge 0.6\rho.
$
Using $a_0\ge 0.997\rho$ and $|u_E|\le\rho$, it is enough that
$
(1-t)0.997\rho-t\rho\ge0.6\rho,
$
which holds for $0\le t\le0.19$.
Hence
\[
K\supset
T\supset
\left\{(u,v):|u|\le0.6\rho,\ 0\le v\le0.19v_E \right\}.
\]
Finally,
\[
0.19v_E
\ge
0.19\cdot0.66\,\rho\sin\beta
=
0.1254\,\rho\sin\beta.
\qedhere
\]
\end{proof}

\subsection{A band attached to a surviving plank}

Here we define a rectangular band $\Sigma_P$ for a surviving plank $P$ in the final family $\cQ$ and show that the band is mostly contained in the new hole. 
Fix a final family $\cQ$.  For $P\in\cQ$, define
\[
 I_P=\{Q\in\cQ: Q\text{ is deep and }\ell_Q\cap\ell_P\text{ is one point in }\rho D\}.
\]
When $I_P\ne\varnothing$, set
\[
 \alpha_P=\max_{Q\in I_P}\ang(P,Q),\qquad h_P=0.07\rho\alpha_P.
\]

\begin{definition}
With the coordinates above, define the band
\[
 \Sigma_P=\{(u,v):|u|\le0.6\rho,\ 0\le v\le h_P\}.
\]
\end{definition}

\begin{figure}[]
    \centering 
\begin{tikzpicture}[
  x=1.25cm,y=1.25cm,
  line cap=round,line join=round,
  every node/.style={font=\small},
  point/.style={circle,fill=black,inner sep=1.45pt},
  axis/.style={->,semithick},
  main/.style={semithick},
]

\def\rad{3}
\def\dP{0.36}        
\def\bandhalf{1.65}  
\def\bandh{1.10}     
\def\uZero{-2.68}    
\def\gammaAng{54}      
\def\plankhalf{0.46} 
\def\plankleft{-0.95}
\def\plankright{4.75} 

\coordinate (F0) at (0,0);
\coordinate (O)  at (0,\dP);
\coordinate (U0) at (\uZero,0);                 
\coordinate (U)  at ($(U0)+(1,0)$);
\coordinate (Rdir) at ($(U0)+(\gammaAng:1)$);

\coordinate (X) at (-1.5370,0.9450);
\coordinate (Y) at ($(U0)!(X)!(Rdir)$);

\fill[black!22] (-\bandhalf,0) rectangle (\bandhalf,\bandh);
\draw[main] (-\bandhalf,0) rectangle (\bandhalf,\bandh);

\begin{scope}[shift={(U0)},rotate=\gammaAng]
  \fill[black!30,opacity=0.65] (\plankleft,-\plankhalf) rectangle (\plankright,\plankhalf);
  \draw[black!55] (\plankleft,-\plankhalf) -- (\plankright,-\plankhalf);
  \draw[black!55] (\plankleft, \plankhalf) -- (\plankright, \plankhalf);
\end{scope}

\begin{scope}
  \clip (-\bandhalf,0) rectangle (\bandhalf,\bandh);
  \begin{scope}[shift={(U0)},rotate=\gammaAng]
    \fill[black!55,opacity=0.85] (\plankleft,-\plankhalf) rectangle (\plankright,\plankhalf);
  \end{scope}
\end{scope}

\begin{scope}[shift={(U0)},rotate=\gammaAng]
  \draw[main] (\plankleft,0) -- (\plankright,0);
\end{scope}
\node[anchor=north east] at (-3.00,-0.70) {$\ell_R$};

\draw[main] (O) circle (\rad);
\node[anchor=west,font=\normalsize] at (-2.92,2.40) {$\rho D$};

\draw[axis] (-3.85,0) -- (3.5,0) node[right] {$u$};
\draw[axis] (0,-3.0) -- (0,4.30) node[above] {$v$};
\node[anchor=east] at (-3.28,0.28) {$\ell_P$};

\node[point] at (F0) {};
\node[point] at (O) {};
\node[below right=1pt] at (F0) {$(0,0)$};
\node[above right=3pt] at (O) {$0$};
\draw[densely dotted] (O) -- (F0);

\node[point] at (U0) {};
\node[below right=0pt] at (U0) {$(u_0,0)$};
\pic[draw,angle radius=6.5mm,angle eccentricity=1.35,"$\gamma$"]
  {angle=U--U0--Rdir};

\node[point] at (X) {};
\node[point] at (Y) {};
\node[below right=2pt] at (X) {$x$};
\node[above left=0pt] at (Y) {$y$};
\draw[densely dotted] (X) -- (Y);

\end{tikzpicture}
    \caption{An illustration for the proof of Lemma~\ref{lem:surviving-band}. The horizontal shaded rectangle is the band $\Sigma_P$, and the other one represents a plank $R$ with center-line~$\ell_R$. The lines $\ell_P$ and $\ell_R$ intersect at $(u_0,0)$ with angle $\gamma$. The point $y$ is the orthogonal projection of $x\in R\cap\Sigma_P$ onto $\ell_R$.}
    \label{fig:surviving-band}
\end{figure}

\begin{lemma}[the band is mostly contained in the new hole]\label{lem:surviving-band}
Let $\cQ$ be a final family of planks. 
For every central $P \in \cQ$ with $I_P\ne\varnothing$, 
\[
 |\Sigma_P\setminus H'|\le0.004|\Sigma_P|,
 \qquad |\Sigma_P|=0.084\rho^2\alpha_P.
\]
\end{lemma}

\begin{proof}
Choose $Q\in I_P$ with $\ang(P,Q)=\alpha:=\alpha_P$.
Since $(P,Q)$ is a candidate pair and $\cQ$ is final, it is not removable.
By Lemma~\ref{lem:hull-lower} and $\sin\alpha\ge 2\alpha/\pi$,
\[
\Sigma_P\subset K(P,Q).
\]

We next claim that every $R\in\cQ$ meeting $\Sigma_P$ belongs to
$N(P,Q)$. Indeed, let $x=(u_x,v_x)\in R\cap\Sigma_P$;
see Figure \ref{fig:surviving-band}.
Because $x \in\Sigma_P$, 
\[
|u_x|\le 0.6\rho,
\qquad
0\le v_x\le h_P=0.07\rho\alpha <0.11\rho.
\]
Because $P$ is central, $d_P\le 0.07\rho$.
Using \eqref{eq:pruning-assumptions}, 
\[
\|x\|+\frac w2
\leq \sqrt{u_x^2+(v_x+d_P)^2} + \frac{w}{2} 
<\rho\sqrt{0.6^2+0.18^2} + \frac{w}{2} <0.64\rho,
\]
so $R$ is deep.
It remains (in fact suffices) to prove that
$\gamma:= \ang(R,P) < 2 \alpha$. Otherwise, 
$\gamma  \geq 2 \alpha$.
Let $y=(u_y,v_y)$ be the orthogonal projection of $x$ onto $\ell_R$.
Since $x=(u_x,v_x)\in R\cap\Sigma_P$, 
$\|x-y\|\le \frac{w}{2}.$
Consequently,
\[
|u_y|
\le |u_x|+\frac{w}{2}
\le 0.6\rho+\frac{w}{2}
\qquad
\text{and}
\qquad
|v_y|
\le v_x+\frac{w}{2}
\le 0.07\rho\alpha+\frac{w}{2}.
\]

Let $(u_0,0)=\ell_R\cap\ell_P.$ Since $\ell_R$ makes angle $\gamma$ with the $u$-axis
$\ell_P$, 
and $\gamma\ge 2\alpha$, in particular
$\cot\gamma\le 1/\alpha$,
we have
\[
|u_0-u_y|
=
|v_y|\cot\gamma
\le
\frac{0.07\rho\alpha+w/2}{\alpha}.
\]
Therefore,
\[
\begin{aligned}
|u_0|
&\le |u_y|+|u_0-u_y|\\
&\le
0.6\rho+\frac{w}{2}
+\frac{0.07\rho\alpha+w/2}{\alpha}\\
&<0.997\rho.
\end{aligned}
\]

Since $P$ is central, $d_P\le 0.07\rho$. Hence the
distance from the origin 0 to the point $(u_0,0)\in\ell_P$ satisfies
\[
\sqrt{u_0^2+d_P^2}
<
\rho\sqrt{0.997^2+0.07^2}
<\rho.
\]
Thus
$\ell_R\cap\ell_P\in \rho D.$
We have already shown that $R$ is deep, so $R\in I_P$. By the
definition of
$\alpha$
it follows that
$\gamma\le \alpha,$
contradicting the assumption $\gamma\ge 2\alpha$. Therefore
$\gamma<2\alpha.$

Because $(P,Q)$ is not removable and $H'=int(D)\setminus\bigcup\cQ$,
we have
\[\Sigma_P \setminus H'
\subset K(P,Q) .\]
It follows that
\[\Sigma_P\setminus H'
= \Sigma_P\cap \bigcup \cQ
\subset K(P,Q)
\cap \bigcup N(P,Q).\]
Hence,
\[
|\Sigma_P\setminus H'|
\le \tau |K(P,Q)|
\le \tau(2\rho^2\alpha+10\rho w)
\le 2.2\tau\rho^2\alpha.
\]
Finally,
\[
|\Sigma_P|=(1.2\rho)(0.07\rho\alpha)
=0.084\rho^2\alpha,
\]
and
\[
|\Sigma_P\setminus H'|
\le
\frac{2.2\tau}{0.084} |\Sigma_P|
<0.004 |\Sigma_P|.
\qedhere
\]
\end{proof}

\subsection{Cells}
Here we define the cells of the new hole and describe their properties.

\begin{definition}
A \emph{cell} is a connected component of $H'$.
\end{definition}

\begin{fact}\label{fact:cell-convex}
Every cell is open and convex.
\end{fact}

\begin{figure}[]
    \centering 
\begin{tikzpicture}[scale=0.7,line join=round,line cap=round]
  \tikzset{
    celloutline/.style={draw=black,line width=0.9pt},
    disk/.style={draw=black,line width=0.9pt},
    point/.style={circle,fill=black,inner sep=1.45pt}
  }

  \begin{scope}[shift={(0,0)}]
    \node[] at (0,3.10) {bad cell $G$};

    \def\rad{2.12}

    \path[fill=black!10]
      (-1.63,0.72)--(-0.43,1.10)--(1.60,1.02)--(2.75,0.62)--
      (3.08,0.08)--(1.95,-0.56)--(-0.23,-0.72)--(-1.33,-0.22)--cycle;

    \begin{scope}
      \clip (0,0) circle[radius=\rad];
      \path[fill=black!25]
        (-1.63,0.72)--(-0.43,1.10)--(1.60,1.02)--(2.75,0.62)--
        (3.08,0.08)--(1.95,-0.56)--(-0.23,-0.72)--(-1.33,-0.22)--cycle;
    \end{scope}

    \draw[celloutline]
      (-1.63,0.72)--(-0.43,1.10)--(1.60,1.02)--(2.75,0.62)--
      (3.08,0.08)--(1.95,-0.56)--(-0.23,-0.72)--(-1.33,-0.22)--cycle;
    \draw[disk] (0,0) circle[radius=\rad];

    \node at (-1.7,1.9) {$rD$};
    \node at (2.75,1.15) {$G$};
    \node[point] at (0,0) {};
    \node[below right=0pt] at (0,0) {$0$};

    \node[font=\large] at (0,-2.95)
      {$|G|<2\lambda|G\cap rD|$};
  \end{scope}

  \begin{scope}[shift={(8.0,0)}]
    \node[] at (0,3.10) {good cell $G$};

    \def\rad{2.12}

    \path[fill=black!10]
      (0.45,0.22)--(1.18,0.92)--(4.02,1.34)--(4.98,1.06)--
      (5.12,0.22)--(4.50,-0.66)--(1.62,-0.40)--cycle;

    \begin{scope}
      \clip (0,0) circle[radius=\rad];
      \path[fill=black!25]
        (0.45,0.22)--(1.18,0.92)--(4.02,1.34)--(4.98,1.06)--
        (5.12,0.22)--(4.50,-0.66)--(1.62,-0.40)--cycle;
    \end{scope}

    \draw[celloutline]
      (0.45,0.22)--(1.18,0.92)--(4.02,1.34)--(4.98,1.06)--
      (5.12,0.22)--(4.50,-0.66)--(1.62,-0.40)--cycle;
    \draw[disk] (0,0) circle[radius=\rad];

    \node at (-1.7,1.9) {$rD$};
    \node at (5.15,1.58) {$G$};
    \node[point] at (0,0) {};
    \node[below right=0pt] at (0,0) {$0$};

    \node[font=\large] at (0,-2.95)
      {$|G|\ge 2\lambda|G\cap rD|$};
  \end{scope}
\end{tikzpicture}
\caption{An illustration of bad and good cells.}
\label{fig:good-bad-cells}
\end{figure}

\begin{definition}
Let $\mathcal G$ be the cells of $H'$ meeting $rD$.  Call $G\in\mathcal G$ \emph{good} if
\[|G|\ge2\lambda|G\cap rD|\]
where
\[ \lambda:=\frac{\rho}{200r}\ge 1,\]
and \emph{bad} otherwise;
see Figure \ref{fig:good-bad-cells}.
\end{definition}

\begin{lemma}[cell properties]\label{lem:vertex}
Let $G$ be a cell with $\overline G\subset int(D)$. Then $G$ is the interior of a convex polygon.  Each edge of the polygon lies on a boundary line of a plank and $G$ lies on the side away from that plank.  If two edges meeting at a vertex $v$ belong to planks $P_1,P_2$ and the exterior angle is $\phi\in(0,\pi)$ and their center lines meet at a point $c$ then
\begin{equation}\label{eq:vertex-geometry}
 \norm{c-v}=\frac{w/2}{\cos(\phi/2)},\qquad
 \ang(P_1,P_2)=\min\{\phi,\pi-\phi\}.
\end{equation}
\end{lemma}
\begin{proof}
Since $\overline G\subset int(D)$, no circular arc of $\partial D$ occurs on the boundary of $G$. Convexity thus gives the polygon statement.  If $n_1,n_2$ are the outward unit normals of the two edges, then $\langle n_1,n_2\rangle=\cos\phi$ and the two center lines are
\[
 \ell_{P_i}=\{x:\langle x-v,n_i\rangle=w/2\}.
\]
Their intersection is
\[
 c=v+\frac w2\frac{n_1+n_2}{1+\cos\phi},
\]
which gives the distance formula.  
\end{proof}

\begin{lemma}[structure of bad cells]\label{lem:bad-cells}
Let $\cQ$ be a final family of planks. 
Assume that for every central $P\in\cQ$ such that $I_P\ne\varnothing$,
\begin{align}
\label{eqn:alpha*}
 \alpha_P\le\alpha_*:=0.19\frac r\rho\le10^{-3}.   
\end{align}
Then,
\begin{enumerate}[label=(\alph*)]
\item Every bad cell lies in $0.06\rho D$, and every plank carrying one of its edges is central.
\item Any two central planks in $\cQ$ make angle at most $0.3$.
\item Every bad cell $G$ has a bounding plank $P(G)$, central, with $I_{P(G)}\ne\varnothing$, such that
\[
 G\subset\{x:\dist(x,\ell_{P(G)}) < 2 h_{P(G)}\}.
\]
\end{enumerate}
\end{lemma}
\begin{proof}

For (a), suppose that a bad cell $G$ contains a point $v$ with
$\norm v\ge12\lambda r$. By Lemma~\ref{lem:convex-growth},
\[
 |G|\ge\frac{(11\lambda r)^2}{4r(12\lambda r)}|G\cap rD|
 \ge2.52\lambda|G\cap rD|,
\]
because the function $t\mapsto (t-r)^2/(4rt)$ is increasing for
$t>r$. This contradicts the definition of a bad cell, so every $v\in G$ satisfies
$\norm v<12\lambda r=0.06\rho,$
and hence
$G\subset0.06\rho D.$

If a plank $P$ carries an edge of $G$, choose a point $z$ on that edge. Then $P$ is central since
\[
d_P
\le \|z\|+\dist(z,\ell_P)
\le 0.06\rho+\frac{w}{2}
<0.07\rho.
\]

For (b), assume towards a contradiction that two central planks made an angle $\mu>0.3$. Their center lines meet within
$\frac{0.14\rho}{\sin0.3}<0.5\rho$
of the origin. So, they lie in each other's $I$-sets. Now, \eqref{eqn:alpha*} implies
$\mu\le\alpha_*<0.3$, a contradiction.

For (c), we claim that every bad cell has exactly two vertices whose exterior angle is at least $\pi-\alpha_*$ and all other exterior angles are at most $\alpha_*$; call the two exceptional vertices the \emph{tips}; see Figure \ref{fig:bad-cell-tips}. Let $v$ be a vertex of a bad cell and $\phi$ its exterior angle.  By (a), $\norm v\le0.06\rho$, and the incident planks are central.  If
$\cos(\phi/2)\ge w/(1.8\rho)$, Lemma \ref{lem:vertex} puts the intersection of their center lines within $0.96\rho$ of the origin, so
\[
 \min\{\phi,\pi-\phi\}\le\alpha_*.
\]
Otherwise, 
\[
 \phi>
 \pi-\frac{2.01 w}{1.8\rho}
 \ge\pi-0.012 \frac{r}{\rho} 
 >\pi-\alpha_*.
\]
In both cases, every vertex is either a tip or has exterior angle at most $\alpha_*$.  Three tips would have total exterior angle greater than $2\pi$, so there are at most two.  Suppose instead that there is at most one tip.  Then the non-tip exterior angles, each at most $\alpha_*$, would have total sum greater than $\pi$.  Traverse the polygon and accumulate these turning angles from one fixed edge. At the first time the accumulated turn reaches $\pi/2$, it lies in $[\pi/2,\pi/2+\alpha_*]$.  The corresponding two unoriented edge lines therefore make an angle in $[\pi/2-\alpha_*,\pi/2]$.  Their carrying planks are central by (a), contradicting (b).  Hence every bad cell has exactly two tips.

Now, let the tip interior angles be
$0<\theta_1\le\theta_2\le\alpha_*.$
Let $P$ and $P'$ be the two planks carrying the two edges incident
to the tip with the angle $\theta_2$, and set $P(G)=P$. Since $\diam G\le0.12\rho$, every point of $G$
lies within $0.12\rho\theta_2$ of each of these two edge lines.

The center lines $\ell_P$ and $\ell_{P'}$ meet at distance at most
\[
1.001\frac{w}{\theta_2}\le0.501\rho
\]
from the tip, hence inside $\rho D$. Thus $P'\in I_P$, so
$I_P\ne\varnothing$ and
$\alpha_P\ge\ang(P,P')=\theta_2.$
Therefore, for $x\in G$,
\[
\dist(x,\ell_P)
\le \frac w2+0.12\rho\theta_2
\le0.13\rho\alpha_P
<2h_P.
\]
Hence
$G\subset\{x:\dist(x,\ell_P)<2h_P\}.$
\end{proof}

\begin{figure}[]
    \centering 

\begin{tikzpicture}[
  x=1cm,y=1cm,
  line cap=round,line join=round,
  cell/.style={line width=1.0pt},
  extension/.style={line width=0.72pt},
  plankline/.style={line width=0.82pt},
  center/.style={densely dashed,line width=0.65pt},
  tip/.style={circle,fill=black,inner sep=2.2pt},
  minor/.style={circle,draw=black,fill=black,line width=0pt,inner sep=0pt}
]

\coordinate (T1) at (-5.05, 0.12);
\coordinate (A1) at (-3.72, 0.66);
\coordinate (A2) at (-1.15, 0.83);
\coordinate (A3) at ( 0.78, 0.56);
\coordinate (T2) at ( 2.52, 0.03);
\coordinate (B1) at ( 1.86,-0.33);
\coordinate (B2) at ( 0.14,-0.74);
\coordinate (B3) at (-2.82,-0.48);

\fill[black!3] (T1)--(A1)--(A2)--(A3)--(T2)--(B1)--(B2)--(B3)--cycle;


\coordinate (UpOff)  at (0.27,0.89);
\coordinate (UpHalf) at (0.135,0.445);
\coordinate (U1) at ($(A3)!-0.95!(T2)$);
\coordinate (U2) at ($(A3)!2.75!(T2)$);
\coordinate (U3) at ($(U2)+(UpOff)$);
\coordinate (U4) at ($(U1)+(UpOff)$);

\coordinate (LoOff)  at (0.47,-0.86);
\coordinate (LoHalf) at (0.235,-0.43);
\coordinate (L1) at ($(T2)!-4.40!(B1)$);
\coordinate (L2) at ($(T2)!2.80!(B1)$);
\coordinate (L3) at ($(L2)+(LoOff)$);
\coordinate (L4) at ($(L1)+(LoOff)$);

\fill[gray!18] (U1)--(U2)--(U3)--(U4)--cycle;
\fill[gray!18] (L1)--(L2)--(L3)--(L4)--cycle;

\begin{scope}
  \clip (U1)--(U2)--(U3)--(U4)--cycle;
  \fill[gray!48] (L1)--(L2)--(L3)--(L4)--cycle;
\end{scope}

\draw[plankline] (U1)--(U2);
\draw[plankline] (U4)--(U3);
\draw[center] ($(U1)+(UpHalf)$)--($(U2)+(UpHalf)$);
\node[font=\small] at (1.35,1.13) {$P'$};

\draw[plankline] (L1)--(L2);
\draw[plankline] (L4)--(L3);
\draw[center] ($(L1)+(LoHalf)$)--($(L2)+(LoHalf)$);
\node[font=\small] at (1.65,-1.30) {$P$};

\coordinate (C) at (3.7485,0.1419);
\fill (C) circle (1.1pt);

\draw[cell] (T1)--(A1)--(A2)--(A3)--(T2)--(B1)--(B2)--(B3)--cycle;
\node[font=\large] at (-1.05,-0.02) {$G$};

\node[] at (-4.30,0.15) {$\theta_1$};
\node[] at ( 1.92,-0.01) {$\theta_2$};


\coordinate (T1ray) at ($(B3)!1.45!(T1)$);
\draw[extension] (T1)--(T1ray);
\pic[draw=black,line width=0.78pt,angle radius=5.0mm]
  {angle=A1--T1--T1ray};

\coordinate (T2ray) at ($(A3)!1.95!(T2)$);
\draw[extension] (T2)--(T2ray);
\pic[draw=black,fill=white,line width=0.78pt,angle radius=4.4mm]
  {angle=B1--T2--T2ray};
\draw[cell] (B1)--(T2);
\draw[extension] (T2)--(T2ray);

\coordinate (A1ray) at ($(T1)!1.72!(A1)$);
\coordinate (A2ray) at ($(A1)!1.55!(A2)$);
\coordinate (A3ray) at ($(A2)!1.60!(A3)$);
\coordinate (B1ray) at ($(T2)!2.15!(B1)$);
\coordinate (B2ray) at ($(B1)!1.65!(B2)$);
\coordinate (B3ray) at ($(B2)!1.60!(B3)$);

\foreach \V/\R in {A1/A1ray,A2/A2ray,A3/A3ray,B1/B1ray,B2/B2ray,B3/B3ray}{
  \draw[extension] (\V)--(\R);
}

\pic[draw=black,fill=black,line width=0.45pt,angle radius=4.9mm]
  {angle=A2--A1--A1ray};
\pic[draw=black,fill=black,line width=0.45pt,angle radius=4.7mm]
  {angle=A3--A2--A2ray};
\pic[draw=black,fill=black,line width=0.45pt,angle radius=4.8mm]
  {angle=T2--A3--A3ray};
\pic[draw=black,fill=black,line width=0.45pt,angle radius=5.0mm]
  {angle=B2--B1--B1ray};
\pic[draw=black,fill=black,line width=0.45pt,angle radius=4.8mm]
  {angle=B3--B2--B2ray};
\pic[draw=black,fill=black,line width=0.45pt,angle radius=4.9mm]
  {angle=T1--B3--B3ray};

\node[tip] at (T1) {};
\node[tip] at (T2) {};

\end{tikzpicture}
    \caption{An illustration for the proof of Lemma~\ref{lem:bad-cells}(c). 
    The bad cell $G$ has exactly two tips, which are marked by black dots.
    The small exterior angles ($\leq \alpha_*$) are marked black,
    while the two large exterior angles ($\geq \pi-\alpha_*$) are marked white.
    The two tips have interior angles $0<\theta_1\le\theta_2\le\alpha_*$.
    The planks $P$ and $P'$ carry the two edges incident to the tip with angle $\theta_2$. Their center lines are  dashed.}
    \label{fig:bad-cell-tips}
\end{figure}

\subsection{Vitali-type selection}
Here we describe a Vitali-type selection
of the bad cells.
Assuming~\eqref{eqn:alpha*},
use Lemma~\ref{lem:bad-cells}(c)
to define the bad planks as
\[
 \bad:=\{P(G):G\text{ bad}\} .
\]
All members of $\bad$ are central by (a), and any two make an angle at most $0.3$ by Lemma~\ref{lem:bad-cells}(b). Rotate coordinates so one is horizontal and write for each $P \in \bad$,
\[
 \ell_P=\{(x,y_P(x))\},\qquad y_P(x)=a_Px+b_P,
\]
with
\[
 |a_P|\leq \tan 0.3 <0.31,\qquad |b_P|=d_P \sqrt{1+ a_P^2}\le0.08\rho.
\]

Recall that the band $\Sigma_P$ is defined in the local coordinates associated with $P$, whose axis is $\ell_P$. 
Now it will be convenient to consider the bands defined in the common coordinate system.

\begin{definition}
The non-local band (i.e., not in the local coordinate system) of $P$ of scale $S > 0$ is
\[
 B_P^S=\{(x,y):|x|\le0.5\rho,\ |y-y_P(x)|\le S h_P\}.
\]
The main band is $B_P=B_P^{1.05}$.
The truncated local band is
\[
 \Sigma_P'=\Sigma_P\cap 0.5\rho D;
\]
see Figure~\ref{fig:bands}.
\end{definition}
\begin{claim}[relating local to non-local bands]
\label{clm:SandB}
For $P \in \bad$, we have
$\Sigma_P'\subset B_P$
and
$|\Sigma_P'|\ge0.98\rho h_P$.
\end{claim}
\begin{proof}
Every point of $\Sigma_P$ is at Euclidean distance at most $h_P$ from $\ell_P$.  Since $\ell_P$ has an angle at most $0.3$ with the $x$-axis, its vertical distance from the graph $y_P$ is at most $h_P/\cos0.3<1.05h_P$.  Hence $\Sigma_P'\subset B_P$.  Moreover, in the local $(u,v)$ coordinates the Euclidean origin is at $(0,\pm d_P)$. The rectangle $\{|u|\le0.49\rho , 0\le v\le h_P\}$ lies in $0.5\rho D$ because $d_P\le0.07\rho$ and $h_P
\le 0.07\cdot 10^{-3}\rho
<10^{-4}\rho$.
\end{proof}

\begin{figure}[]
    \centering 

\begin{tikzpicture}[x=1cm,y=1cm,line cap=round,line join=round,
                    every node/.style={font=\small}]
  \def\R{3}
  \def\CY{0.35} 

  \coordinate (BL) at (-3.00,-1.2515);
  \coordinate (BR) at ( 3.00,-0.4085);
  \coordinate (TR) at ( 3.00, 0.3515);
  \coordinate (TL) at (-3.00,-0.4915);

  \coordinate (SL)  at (-3.5030,-0.9423);
  \coordinate (SR)  at ( 3.6270, 0.0597);
  \coordinate (SRt) at ( 3.5936, 0.2974);
  \coordinate (SLt) at (-3.5364,-0.7046);

  \fill[gray!18] (BL)--(BR)--(TR)--(TL)--cycle;
  \fill[gray!18] (SL)--(SR)--(SRt)--(SLt)--cycle;

  \begin{scope}
    \clip (BL)--(BR)--(TR)--(TL)--cycle;
    \fill[gray!52] (SL)--(SR)--(SRt)--(SLt)--cycle;
  \end{scope}

  \draw[line width=0.9pt] (0,\CY) circle (\R);

  \draw[dotted,gray!80,line width=0.8pt] (-\R,\CY-1.9)--(-\R,\CY+1.9);
  \draw[dotted,gray!80,line width=0.8pt] ( \R,\CY-1.9)--( \R,\CY+1.9);

  \draw[dashed,line width=0.9pt] (BL)--(BR)--(TR)--(TL)--cycle;

  \draw[line width=0.75pt] (SL)--(SR)--(SRt)--(SLt)--cycle;

  \draw[line width=0.8pt] (-4.15,-1.033)--(4.20,0.140);

  \fill (0,\CY) circle (1.35pt);
  \node[above right=0pt] at (0,\CY) {$0$};

  \node at (-1.6,2.3) {$0.5\rho D$};
  \node[below left] at (-3.93,-0.99) {$\ell_P$};
  \node at (0.10,-1.15) {$B_P$};
  \node[above right] at (3.37,0.23) {$\Sigma_P$};
\end{tikzpicture}
    \caption{An illustration of the bands. The circle is $0.5\rho D$. The local band $\Sigma_P$ is the light gray rectangle, the non-local band $B_P$ is the parallelogram with dashed boundary, and their intersection is shown in dark gray. Claim~\ref{clm:SandB} shows that
    $\Sigma_P':=\Sigma_P\cap 0.5\rho D\subset B_P$.}
    \label{fig:bands}
\end{figure}

\begin{lemma}[two possibilities for crossing]\label{lem:carrier-lines}
For every $P,Q\in\bad$,
\begin{enumerate}[label=(\arabic*)]
\item If $\ell_P,\ell_Q$ meet at an abscissa $|x|\le0.85\rho$, then
\[
 0.5\rho|a_P-a_Q|\le7.9\min\{h_P,h_Q\}.
\]
\item Otherwise $|b_P-b_Q|\ge0.85\rho|a_P-a_Q|$, and for $|x|\le0.5\rho$,
\[
 0.41|b_P-b_Q|\le|y_P(x)-y_Q(x)|\le1.59|b_P-b_Q|.
\]
\end{enumerate}
\end{lemma}
\begin{proof}
In (1), the intersection point has vertical coordinate at most
$0.08\rho+0.85\rho\cdot0.31<0.35\rho$, so it lies in $\rho D$.  Therefore,
$\mu:=\ang(P,Q)\le\min\{\alpha_P,\alpha_Q\}$.  Since both directions lie within $0.3$ of the horizontal,
for
$a_P = \tan \theta_P$
and
$a_Q = \tan \theta_Q$,
\[
|a_P-a_Q|
=
|\tan\theta_P-\tan\theta_Q|
=
\frac{|\sin(\theta_P-\theta_Q)|}{|\cos\theta_P\cos\theta_Q|}
\le
\frac{\mu}{\cos^2 0.3}
\le 1.1\mu.
\]
This proves (1) because $h_P=0.07\rho\alpha_P$.

For (2), either the lines are parallel or their intersection abscissa $x$ has absolute value greater than $0.85\rho$. 
Because $(a_P-a_Q)x = b_Q-b_P$, the first assertion holds.
For $|x| \leq 0.5 \rho$, we have
\[
|y_P(x)-y_Q(x)|
\ge |b_P-b_Q|-|(a_P-a_Q)x|
\ge (1-0.59)|b_P-b_Q|
=0.41|b_P-b_Q|;
\]
the same argument gives the upper bound.
\end{proof}

\begin{lemma}[Vitali-type selection]\label{lem:band-selection}
Assuming~\eqref{eqn:alpha*},
there is $\cV\subset\bad$ such that the following hold:
\begin{enumerate}[label=(\Alph*)]
\item The bands $B_j$, $j\in\cV$, are pairwise disjoint.
\item For every $P\in\bad$, there is $j\in\cV$ with $h_j\ge h_P$ and
\[
 \{(x,y):|x|\le0.5\rho,\ \dist((x,y),\ell_P)\le 2h_P\}\subset B_j^{24}.
\]
\end{enumerate}
\end{lemma}
\begin{proof}
Order $\bad$ by decreasing $h_P$ and select $P$ unless an already selected $j$ satisfies
$|b_P-b_j|\le12h_j$.  

For (A), if $P,j$ are selected with $h_j \geq h_P$ then
Lemma \ref{lem:carrier-lines} gives
\[
|y_P(x)-y_j(x)|
\ge |b_P-b_j|-|a_P-a_j||x|
>12h_j-7.9h_P
\ge 4.1h_j
\]
in case $(1)$, and
\[
|y_P(x)-y_j(x)|
\ge 0.41|b_P-b_j|
>0.41\cdot 12\,h_j
=4.92h_j
>4.9h_j
\]
in case (2).
Both lower bounds are larger than $1.05(h_j+h_P)$.  

It remains to prove (B). If $P$ is rejected, choose the earlier $j$ with $h_j\ge h_P$ and $|b_P-b_j|\le12h_j$.  Lemma \ref{lem:carrier-lines} gives when $|x| \leq 0.5 \rho$,
\[
|y_P(x)-y_j(x)|
\le |b_P-b_j|+|a_P-a_j||x|
\le 12h_j+7.9h_P
\le 19.9h_j
\]
in case (1) and
\[
|y_P(x)-y_j(x)|
\le 1.59|b_P-b_j|
\le 1.59\cdot 12h_j
<19.1h_j
\]
in case (2).  
If in addition $\dist((x,y),\ell_P)\le 2h_P$
then, for $a_P=\tan\theta$,
\[
|y-y_P(x)|
\le \frac{2h_P}{\cos\theta}
\le \frac{2h_P}{\cos 0.3}
\le 2.2h_j .
\]
The total is below $24h_j$.
\end{proof}

\subsection{Proof of Theorem \ref{thm:pruning}}
First, assume that~\eqref{eqn:alpha*}
does not hold.
That is, there is a central $P \in \cQ$
with $I_P \neq \emptyset$ such that
$\alpha_P > \alpha_*:=0.19\frac r\rho$.
By Lemma~\ref{lem:surviving-band}, we are done:
\[
 |H'| \geq |\Sigma_P \cap H'|
 \geq 0.996  \cdot 0.084\rho^2
 \cdot 0.19\frac r\rho
 \geq \lambda \pi r^2
 \geq \lambda |H| .
\]
So, assume that~\eqref{eqn:alpha*} holds.
By Lemma~\ref{lem:band-selection},
there is a Vitali cover $\cV$ of the bad cells;
if there are no bad cells, $\cV$ is empty.
For the good cells, by definition,
\[
\sum_{G\ \mathrm{good}}|G\cap rD|
\le \frac{1}{2\lambda}\sum_{G\ \mathrm{good}}|G|
\le \frac{|H'|}{2\lambda}
=100\frac r\rho|H'|.
\]
For the bad cells, Lemmas~\ref{lem:bad-cells} and~\ref{lem:band-selection} imply
\[
\bigcup_{G\ \mathrm{bad}}G
\subseteq \bigcup_{j\in\mathcal V}B_j^{24}.
\]
Because the vertical sections of \(B_j^{24}\) have length \(48h_j\),
\[
\sum_{G\ \mathrm{bad}}|G\cap rD|
\le 96r\sum_{j\in\mathcal V}h_j.
\]
By Claim~\ref{clm:SandB} and
Lemma~\ref{lem:surviving-band}, for every $j \in \cV$,
\begin{align*}
|\Sigma'_j\cap H'|
&\ge |\Sigma'_j|-|\Sigma_j\setminus H'|\\
&\ge 0.98\rho h_j-0.004\cdot1.2\rho h_j =0.9752\rho h_j.
\end{align*}
Summing gives
\[
0.9752\rho\sum_{j\in\mathcal V}h_j\le |H'|.
\]
Consequently,
\[
\sum_{G\ \mathrm{bad}}|G\cap rD|
\le \frac{96}{0.9752}\frac r\rho|H'|
\le 99\frac r\rho|H'|.
\]
Finally, since \(H\subseteq H'\cap rD\),
\[
|H|
\le
\sum_{G\ \mathrm{good}}|G\cap rD|
+
\sum_{G\ \mathrm{bad}}|G\cap rD|
\le 199\frac r\rho|H'|.
\]
\qed

\section{Removable pairs pay overlap}
\label{sec:removablePair}
In this section, we translate overlap to geometry.
We explain how a removable pair leads to a dense region that contributes to overlap. 

\begin{definition}
Call a disk $B \subset \R^2$ \emph{$\delta$-dense} if $B\subset D$ and
\[ \int_B(M_{\cP}(x)-1)_+\,dx>\delta |B|.
\]
The \emph{dense region} $U_\delta$ is the union of all such $\delta$-dense disks; see Figure \ref{fig:dense-region}.
\end{definition}

\begin{figure}[]
    \centering 
    \input{dense-region.tex}
    \caption{An illustration of a set of planks and the corresponding dense region $U_\delta$.}
    \label{fig:dense-region}
\end{figure}

The main assertion of this section is the following proposition. 
Choose constants
\[
 \delta=10^{-4},\qquad 
 A=10^{13},\qquad \gamma=10^{-14}.
\]

\begin{proposition}[removable pair pays overlap]\label{prop:pair-payment}
Let $P$ be central and $Q$ deep, let
$\ell_P\cap\ell_Q=\{c\}\subset\rho D$, and put $\beta=\ang(P,Q)>0$.  Suppose
\begin{equation}\label{eq:thin-A}
 \beta\ge A\frac w\rho
\end{equation}
and $(P,Q)$ is removable.  Then,
\begin{equation}\label{eq:pair-payment}
 |U_{\delta}\cap(P\cup Q)\cap5\rho D|\ge \gamma\rho w.
\end{equation}
\end{proposition}

The rest of this section is devoted to the proof of the proposition. 
\subsection{Reduction to a triangle}

Let $J_P=\ell_P\cap\rho  D$, $J_Q=\ell_Q\cap\rho D$, and
$K_0=\conv(J_P\cup J_Q)$.  The two chords divide $K_0$ into four triangles with apex $c$.  Since
$|J_P|\ge1.99\rho$ and $|J_Q|\ge1.32\rho$,
\begin{equation}\label{eq:K0-lower}
 |K_0|=\frac12|J_P||J_Q|\sin\beta\ge0.83\rho^2\beta.
\end{equation}

\begin{lemma}[choosing triangle]\label{lem:triangle-reduction}
Let $\eta=\tau/40=2.5\cdot10^{-6}$.  Under the hypotheses of Proposition \ref{prop:pair-payment}, there exist a subfamily

\[\cR\subset N'(P,Q):=N(P,Q)\setminus\{P,Q\},\] one distinguished arm $\ell_*\in\{\ell_P,\ell_Q\}$, and unit vectors $e_0,e_1$ along the two center lines of $P,Q$ such that
the following hold:
\begin{enumerate}[label=(\roman*)]
\item Every $R\in\cR$ makes angle at most $2\beta$ with $\ell_*$.
\item With
\[
 T=\conv\{c,c+2\rho e_0,c+2\rho e_1\},
\]
we have
\[
\Big|T\cap\bigcup \cR \Big|\ge\eta|T|.
\]
\item The opening angle of $T$ at $c$ is $\beta$ or $\pi-\beta$, and
$|T|=2\rho^2\sin\beta$.
\end{enumerate}
\end{lemma}
\begin{proof}
As in the proof of Lemma \ref{lem:hull-upper},
$K(P,Q)\subset K_0+B(0,1.1w)$, and
\[
 |K(P,Q)\setminus K_0|\le9\rho w,
 \qquad |K(P,Q)\cap(P\cup Q)|\le6\rho w.
\]
Because $(P,Q)$ is removable,  \eqref{eq:K0-lower} implies
\[
 \Big|K_0\cap\bigcup N'(P,Q)\Big|
 \ge\tau|K_0|-15\rho w
 \ge\frac\tau2|K_0|.
\]
Assign each plank in $N'(P,Q)$ to 
the plank between $P,Q$ with a smaller angle. One of the two classes covers at least $(\tau/4)|K_0|$.  One of the four constituent triangles of $K_0$ then receives at least $(\tau/16)|K_0|\ge0.05\tau\rho^2\beta$.  Extending its two radial sides to length $2\rho$ produces $T$. Since $|T|\le2\rho^2\beta$, the covered fraction is at least $\tau/40=\eta$.
\end{proof}

\subsection{The flattening mechanism}

We isolate the analytic-geometric statement used by Proposition \ref{prop:pair-payment}.

\begin{lemma}[main flattening]\label{lem:main-packet}
Let $T$ be the triangle from Lemma \ref{lem:triangle-reduction}. Let $d$ be a direction making angle at most $\psi_0$ with each arm of $T$.  Let $\cR_0$ be a family of width-$w$ planks, disjoint from $\{P,Q\}$, whose center lines make angle at most $\theta$ with $d$.  Suppose
\begin{equation}\label{eq:theta-cond}
 3\theta\le0.1\sin\beta,
\end{equation}
\begin{equation}\label{eq:packet-cond}
 \Big |T\cap\bigcup \cR_0 \Big|\ge\eta_0|T|,
 \qquad
 w\le\frac{\eta_0\rho\sin\beta}{6000},
 \qquad 0<\eta_0\le\frac14.
\end{equation}
Put
\[
 \mu_0=\min \{1,\tan \psi_0\}.
\]
Then, with $\delta_1=10^{-4}$,
\begin{equation}\label{eq:packet-conclusion}
 |U_{\delta_1}\cap(P\cup Q)\cap5\rho D|
 \ge\frac{\eta_0\sin\beta}{5200(\mu_0 + \tan\theta)}\rho w.
\end{equation}
\end{lemma}

\begin{proof}
Discard the planks that do not meet $T$. Use orthonormal coordinates with origin at the crossing point $c$ and horizontal axis in direction $d$
(sets such as $D$, $4\rho D$, and $5\rho D$ are with respect to the original coordinate system).  Each center line of $R \in \cR_0$ is a graph
\[
 y=a_Rx+b_R,\qquad |a_R|\le\tan\theta\le\frac14.
\]
A vertical section of a width-$w$ plank with slope $a_R$ has length $w\sqrt{1+a_R^2}\in[w,1.04w]$.  Since $T\subset B(c,3\rho)$ in the original plane, every plank in $\cR_0$ has a center-line point whose $c$-centered coordinates $(x_1,y_1)$ satisfy $\sqrt{x_1^2+y_1^2}\le3\rho+w/2$.  Hence, for $|x|\le5\rho$ in these $c$-centered coordinates,
\begin{equation}\label{eq:center-height-bound}
|a_Rx+b_R|\le |y_1|+\frac14|x-x_1|
 \le 3\rho+\frac w2+\frac14\left(8\rho+\frac w2\right)
 <5.01\rho.
\end{equation}
The same calculation with $|x|\le6\rho$ gives 
\begin{equation}\label{eq:center-height-bound-six}
 |a_Rx+b_R|\le5.26\rho.
\end{equation}

\medskip
\noindent\emph{Step 1: flatten the packet.}
For fixed $x$, let
\[
 f_1(x)\le\cdots\le f_m(x)
\]
be the ordered center heights of the planks.  Each $f_i$ is continuous and piecewise affine, with every affine slope in $[-\tan\theta,\tan\theta]\subset[-1/4,1/4]$.  Put $s=w/100$ and inductively define
\[
 g_1=f_1,\qquad g_2=f_2,\qquad
 g_i=\max\{f_i,g_{i-2}+s\}\quad(i\ge3);
\]
recall Figure \ref{fig:lipschitz}.
The functions $g_i$ are again continuous and piecewise affine.  Their slopes remain in $[-\tan\theta,\tan\theta]\subset[-1/4,1/4]$. By definition,
\begin{equation}\label{eq:g-separation}
 g_{i+2}\ge g_i+s.
\end{equation}
In addition, for every $p$ and $x$ there is $j\le p$ such that
\begin{equation}\label{eq:g-approx}
 f_p(x)\le g_j(x)<f_p(x)+s.
\end{equation}
Indeed, take the least $j\le p$ with $g_j(x)\ge f_p(x)$.  If $g_j(x)>f_p(x)$, then $g_j=g_{j-2}+s$, and minimality gives $g_{j-2}(x)<f_p(x)$.

Define the ribbons
\[
 L_i=\{(x,y):|y-g_i(x)|<h\},\qquad h=0.54w.
\]
Because every vertical plank section has half-length at most $0.52w$, \eqref{eq:g-approx} implies that the ribbons cover $\bigcup\cR_0$.  On the other hand, \eqref{eq:g-separation} gives ribbon multiplicity at most
\[
 2\left(\frac{2h}{s}+1\right)\le220.
\]

\begin{figure}[]
    \centering 

\begin{tikzpicture}[scale=1.2,
  x=1.35cm,y=1.35cm,
  line cap=round,line join=round,
  every node/.style={font=\small},
  main/.style={line width=0.85pt},
  helper/.style={draw=black!45,line width=0.65pt},
  point/.style={circle,fill=black,inner sep=1.35pt},
  densebd/.style={dotted,line width=0.82pt}
]

\coordinate (C) at (0,0);
\def\angzero{-7}
\def\angone{26}
\def\rr{0.32} 

\coordinate (PLa) at ($(C)+(\angzero:-1.25)$);
\coordinate (PLb) at ($(C)+(\angzero: 3.15)$);
\coordinate (QLa) at ($(C)+(\angone :-1.20)$);
\coordinate (QLb) at ($(C)+(\angone : 3.05)$);

\fill[black!24,fill opacity=0.34]
  ($(PLa)+(\angzero+90:\rr)$)--($(PLb)+(\angzero+90:\rr)$)--
  ($(PLb)+(\angzero-90:\rr)$)--($(PLa)+(\angzero-90:\rr)$)--cycle;
\fill[black!24,fill opacity=0.34]
  ($(QLa)+(\angone+90:\rr)$)--($(QLb)+(\angone+90:\rr)$)--
  ($(QLb)+(\angone-90:\rr)$)--($(QLa)+(\angone-90:\rr)$)--cycle;

\draw[black,line width=0.65pt] (PLa)--(PLb);
\draw[black,line width=0.65pt] (QLa)--(QLb);
\draw[main] (C)--($(C)+(\angzero:3.15)$);
\draw[main] (C)--($(C)+(\angone:3.05)$);

\node[left=15pt,  yshift=2pt, black] at ($(C)+(\angzero:-0.92)$) {$P$};
\node[below left=14pt, yshift=8pt,  black] at ($(C)+(\angone:-0.90)$) {$Q$};

\coordinate (Pfar)   at (-0.82,-0.66);
\coordinate (Pminus) at (-0.10,-0.48);
\coordinate (P)      at ($(C)+(\angzero:1.00)$);
\coordinate (M1)     at (1.28,0.03);
\coordinate (M2)     at (1.82,0.11);
\coordinate (R)      at (2.34,0.19);
\coordinate (Qplus)  at (2.86,0.52);
\def\hh{0.23}

\fill[black!30,fill opacity=0.42]
  ($(Pfar)+(0,\hh)$)--($(Pminus)+(0,\hh)$)--($(P)+(0,\hh)$)--($(M1)+(0,\hh)$)--
  ($(M2)+(0,\hh)$)--($(R)+(0,\hh)$)--($(Qplus)+(0,\hh)$)--
  ($(Qplus)+(0,-\hh)$)--($(R)+(0,-\hh)$)--($(M2)+(0,-\hh)$)--
  ($(M1)+(0,-\hh)$)--($(P)+(0,-\hh)$)--($(Pminus)+(0,-\hh)$)--($(Pfar)+(0,-\hh)$)--cycle;
\draw[black!45,line width=0.50pt]
  ($(Pfar)+(0,\hh)$)--($(Pminus)+(0,\hh)$)--($(P)+(0,\hh)$)--($(M1)+(0,\hh)$)--
  ($(M2)+(0,\hh)$)--($(R)+(0,\hh)$)--($(Qplus)+(0,\hh)$);
\draw[black!45,line width=0.50pt]
  ($(Pfar)+(0,-\hh)$)--($(Pminus)+(0,-\hh)$)--($(P)+(0,-\hh)$)--($(M1)+(0,-\hh)$)--
  ($(M2)+(0,-\hh)$)--($(R)+(0,-\hh)$)--($(Qplus)+(0,-\hh)$);
\draw[line width=0.92pt]
  (Pfar)--(Pminus)--(P)--(M1)--(M2)--(R)--(Qplus);

\node[right=5pt] at (2.73,0.48) {$L_i$};

\pgfmathsetmacro{\zx}{1.26*cos(\angzero)}
\pgfmathsetmacro{\zy}{1.26*sin(\angzero)}
\coordinate (Z) at (\zx,\zy);

\pgfmathsetmacro{\px}{cos(\angzero)}
\pgfmathsetmacro{\py}{sin(\angzero)}
\pgfmathsetmacro{\tt}{(\zx-\px)/(1.28-\px)}
\pgfmathsetmacro{\dzeroY}{\py+\tt*(0.03-\py)}
\coordinate (D0) at (\zx,\dzeroY);

\def\RR{0.50}
\draw[densebd] (D0) circle [radius=\RR];
\node[point] at (D0) {};
\node[above right=3pt] at ($(D0)+(0.22,0.24)$) {$B_0$};

\node[point] at (Z) {};
\node[below right=0pt] at (Z) {$z$};
\draw[dashed] ($(D0)+(0,-0.95)$)--($(D0)+(0,0.28)$);
\node[below=2pt] at ($(D0)+(0,-0.95)$) {$x=x_0$};

\draw[line width=0.82pt] (Z) circle [radius=\rr];

\node[point] at (C) {};
\node[below=1pt] at (C) {$c$};

\end{tikzpicture}
\caption{An illustration for Step 2 in the proof of Lemma~\ref{lem:main-packet}.
The point $z=(x_0,y_0)$ lies in the ribbon $L_i$ and on the center line of the plank $P$. Step~2 shows that $B(z,w/2)$ is contained in the dense region $U_{\delta_1}$. The dotted disk $B_0$ illustrates the first subcase: it is $\delta_1$-dense, is centered at $(x_0,g_i(x_0))$ on the center graph of $L_i$, and contains $B(z,w/2)$.}
\label{fig:ribbon-endpoint}
    
\end{figure}

\medskip
\noindent\emph{Step 2: a ribbon endpoint lies in dense overlap.}
We claim that if
\[
 z=(x_0,y_0)\in L_i\cap4\rho D\cap\ell_S,
 \qquad S\in\{P,Q\},
\]
then
\begin{equation}\label{eq:ribbon-dense}
 B(z,w/2)\subset S\cap U_{\delta_1};
\end{equation}
see Figure \ref{fig:ribbon-endpoint}.
The inclusion in $S$ is immediate because $z$ lies on the center line of the width-$w$ plank $S$.  It remains to find a $\delta_1$-dense disk containing $B(z,w/2)$.

Unwind the recursion defining $g_i$ at $x_0$.  For some integer $q\ge0$,
\[
 g_i(x_0)=f_{i-2q}(x_0)+qs.
\]
Assume first that $q\ge1$.  The $2q+1$ center heights
$f_{i-2q}(x_0),\dots,f_i(x_0)$ lie in an interval of length $qs$.  
There are two sub-cases.

The first sub-case is \[qs\le\rho.\]
As $x$ varies over $|x-x_0|\le qs$, slopes bounded by $1/4$ increase the total spread by at most $\frac12qs$.  Hence, the corresponding center heights lie in an interval of length at most $1.5qs$, and their vertical plank sections lie in an interval of length at most $1.5qs+1.04w$.  On each such vertical line, the one-dimensional excess is at least
\[
 (2q+1)w-(1.5qs+1.04w)\ge qw.
\]
After integrating over $|x-x_0|\le qs$, the excess is at least
\begin{equation}\label{eq:q-excess}
 2qs\cdot qw=0.02q^2w^2.
\end{equation}
In this case, all of this excess lies in
\[
 B_0:=B\bigl((x_0,g_i(x_0)),3(qs+w)\bigr).
\]
Indeed, throughout the integration strip the horizontal displacement is at most $qs$, while the relevant center heights remain within $1.25qs$ of $g_i(x_0)$.  

We now claim that $B_0 \subset D$. 
Because $z\in4\rho D$ and $c\in\rho D$, its $c$-centered coordinates satisfy $\sqrt{x_0^2+y_0^2}=\|z-c\|<5\rho$.  If $qs\le\rho$, then \eqref{eq:center-height-bound} gives
\[
 |g_i(x_0)|=|f_{i-2q}(x_0)+qs|\le6.01\rho.
\]
Thus the center of $B_0$ is at distance less than
$\sqrt{(5\rho)^2+(6.01\rho)^2}<7.82\rho$
from $c$, while its radius is at most $3\rho+3w<3.001\rho$.  Since $\|c\|<\rho$ and $\rho=1/12$, 
we can conclude that $B_0 \subset D$. 

Because
\[
 |B_0|=9\pi w^2(0.01q+1)^2,
\]
we can deduce that
$B_0$ is $\delta_1$-dense.
Since \(z=(x_0,y_0)\in L_i\),
\[
|z-(x_0,g_i(x_0))|<h=0.54w.
\]
For every \(z'\in B(z,w/2)\),
\[
\|z'-(x_0,g_i(x_0))\|
\le \|z'-z\|+\|z-(x_0,g_i(x_0))\|
<1.04w.
\]
It follows that
\[
B(z,w/2)\subset B_0.
\]

The second sub-case is
\[qs>\rho.\] Then, $(2q+1)w>200\rho$.  For every $x$ with $|x-x_0|\le\rho$, we have $|x|<6\rho$ in the $c$-centered coordinates, so \eqref{eq:center-height-bound-six} places every center height in $[-5.26\rho,5.26\rho]$.  Including the plank half-sections, all vertical sections lie in an interval of length less than $10.53\rho$.  There are $2q+1$ planks, so the one-dimensional vertical excess is more than
$(2q+1)w-10.53\rho>189\rho$.
Integrating over $|x-x_0|\le\rho$ gives excess greater than $378\rho^2$.  The contributing region lies in the $c$-centered rectangle
\[
 [-6\rho,6\rho]\times[-5.27\rho,5.27\rho]
 \subset B(c,8\rho)\subset D.
\]
It follows that $B(c,8\rho)$ is $\delta_1$-dense. Finally,  $B(z,w/2)\subset B(c,8\rho)$.

It remains to consider $q=0$.  Then $g_i(x_0)=f_i(x_0)$.  Let $R\in\cR_0$ be a plank whose center line has height $f_i(x_0)$.  Because $z\in L_i$, the distance from $z$ to $\ell_R$ is less than $h=0.54w<5w/8$.  Let $z_R$ be the foot of the perpendicular from $z$ to $\ell_R$ and let $m$ be the midpoint of $[z,z_R]$.  The disks $B(z,w/2)\subset S$ and $B(z_R,w/2)\subset R$ both contain
$B(m,3w/16)$.
Thus, $M_{\cP}\ge2$ on a disk of radius $3w/16$ contained in $B(z,2w)$:
\[
 \int_{B(z,2w)}(M_{\cP}(x)-1)_+
 >\delta_1 |B(z,2w)|.
\]
Since $z\in4\rho D$ and $2w<\rho$, the disk $B(z,2w)$ lies in $D$.  This completes the proof of \eqref{eq:ribbon-dense}.

\begin{figure}[]
    \centering 

\begin{tikzpicture}[scale=1.2,
  x=1.35cm,y=1.35cm,
  line cap=round,line join=round,
  every node/.style={font=\small},
  main/.style={line width=0.85pt},
  helper/.style={draw=black!45,line width=0.65pt},
  point/.style={circle,fill=black,inner sep=1.35pt}
]

\coordinate (C) at (0,0);
\def\angzero{-7}
\def\angone{26}

\coordinate (A) at ($(C)+(\angzero:3.00)$);
\coordinate (B) at ($(C)+(\angone:2.75)$);

\coordinate (A0) at ($(C)+(\angzero:2.28)$);
\coordinate (B0) at ($(C)+(\angone:1.56)$);

\draw[helper] ($(C)+(\angzero:-1.00)$)--($(C)+(\angzero:3.65)$);
\draw[helper] ($(C)+(\angone:-1.00)$)--($(C)+(\angone:3.55)$);
\draw[main] (C)--($(C)+(\angzero:3.65)$);
\draw[main] (C)--($(C)+(\angone:3.55)$);
\node[above left=-1pt, black!70] at ($(C)+(\angzero:-0.82)$) {$\ell_P$};
\node[left=6pt, black!70] at ($(C)+(\angone:-0.75)$) {$\ell_Q$};
\node[below left=0pt] at ($(C)+(\angzero:3.35)$)
  {$\widehat E_0=\ell_P\cap4\rho D$};
\node[above left=-6pt] at ($(C)+(\angone:2.95)$)
  {$\widehat E_1=\ell_Q\cap4\rho D$};

\fill[black!4] (C)--(A)--(B)--cycle;
\draw[dashed,line width=0.82pt] (A0)--(A);
\draw[dashed,line width=0.82pt] (B0)--(B);
\draw[dashed,line width=0.82pt] (A)--(B);
\node at (2.05,0.65) {$\widetilde T$};

\fill[black!14] (C)--(A0)--(B0)--cycle;
\draw[black!55,line width=0.72pt] (A0)--(B0);
\node at (1.14,0.32) {$T$};

\coordinate (Pminus) at (0.42,-0.29);
\coordinate (P)      at ($(C)+(\angzero:1.00)$);
\coordinate (M1)     at (1.28,0.03);
\coordinate (M2)     at (1.82,0.11);
\coordinate (R)      at (2.34,0.19);
\coordinate (Qplus)  at (2.86,0.52);
\def\hh{0.115}


\fill[black!30,fill opacity=0.42]
  ($(Pfar)+(0,\hh)$)--($(Pminus)+(0,\hh)$)--($(P)+(0,\hh)$)--($(M1)+(0,\hh)$)--
  ($(M2)+(0,\hh)$)--($(R)+(0,\hh)$)--($(Qplus)+(0,\hh)$)--
  ($(Qplus)+(0,-\hh)$)--($(R)+(0,-\hh)$)--($(M2)+(0,-\hh)$)--
  ($(M1)+(0,-\hh)$)--($(P)+(0,-\hh)$)--($(Pminus)+(0,-\hh)$)--($(Pfar)+(0,-\hh)$)--cycle;
\draw[black!45,line width=0.50pt]
  ($(Pfar)+(0,\hh)$)--($(Pminus)+(0,\hh)$)--($(P)+(0,\hh)$)--($(M1)+(0,\hh)$)--
  ($(M2)+(0,\hh)$)--($(R)+(0,\hh)$)--($(Qplus)+(0,\hh)$);
\draw[black!45,line width=0.50pt]
  ($(Pfar)+(0,-\hh)$)--($(Pminus)+(0,-\hh)$)--($(P)+(0,-\hh)$)--($(M1)+(0,-\hh)$)--
  ($(M2)+(0,-\hh)$)--($(R)+(0,-\hh)$)--($(Qplus)+(0,-\hh)$);
\draw[line width=0.92pt]
  (Pfar)--(Pminus)--(P)--(M1)--(M2)--(R)--(Qplus);

\draw[line width=0.92pt]
  (Pminus)--(P)--(M1)--(M2)--(R)--(Qplus);

\coordinate (Sleft)  at ($(C)+(\angzero:0.75)$);
\coordinate (Sright) at ($(C)+(\angzero:1.18)$);
\draw[line width=1.8pt] (Sleft)--(Sright);
\node[right=5pt] at (2.73,0.48) {$L_i$};

\node[point] at (P) {};
\node[below=3pt] at (P) {$p$};

\node[point] at (C) {};
\node[below=1pt] at (C) {$c$};

\end{tikzpicture}
    \caption{An illustration for Step 3 in the proof of Lemma~\ref{lem:main-packet}.
The triangle $T$ is contained in the enlarged triangle $\widetilde T$.
The ribbon $L_i$ meets one of the extended arms
$\widehat E_0=\ell_P\cap4\rho D$ at the point $p$ and contains a nontrivial
segment of that arm.}
    \label{fig:ribbon-reaches-an-arm}
\end{figure}

\medskip
\noindent\emph{Step 3: force each relevant ribbon to reach an arm.}
Extend the two arms of $T$ to
\[
 \widehat E_0=\ell_P\cap4\rho D,
 \qquad \widehat E_1=\ell_Q\cap4\rho D;
\]
see Figure \ref{fig:ribbon-reaches-an-arm}.
We choose a slightly larger triangle
\[
 \widetilde T=\conv\{c,c+L'_0e_0,c+L'_1e_1\},
 \qquad L'_0,L'_1\in[2\rho,2.9\rho],
\]
so that $T\subset\widetilde T\subset B(c,2.9\rho)\subset4\rho D$, its perimeter is at most $11.6\rho$, and its third side makes angle at least $1.5\theta$ with the horizontal.
This choice is possible by the continuous choice of $L'_0,L'_1$
using the assumption $0.1\sin\beta\ge3\theta$.

For each ribbon, define
\[
 X_i=\{x:(x,g_i(x))\in\operatorname{int}\widetilde T\}.
\]
Suppose $X_i\ne\varnothing$ and take a maximal component $(x_1,x_2)$ of $X_i$.  Its two center-curve endpoints lie on $\partial\widetilde T$.  The center graph has slopes in $[-\tan\theta,\tan\theta]$.  The third side is either vertical or has slope of absolute value $>\tan\theta$. In both cases, the center curve meets the third side at most once.  At least one endpoint of the component therefore lies on one of the two arms.
Let $p$ be such an endpoint and let $\ell\in\{\ell_P,\ell_Q\}$ be the supporting arm.  
If the arm angle from the horizontal is $\psi\le\psi_0$ and $\tan\psi\le1$, then over a horizontal displacement $t$, the vertical separation between the arm and the center graph grows at rate at most $\tan\psi+\tan\theta\le\mu_0+\tan\theta$.  Hence $L_i$ contains an arm segment through $p$ of length at least $\frac{2h}{\mu_0+\tan\theta}$.
If $\tan\psi>1$, use the vertical coordinate instead; since $|\cot\psi|<1$ and $\mu_0=1$, the same lower bound follows. Therefore,
\begin{equation}\label{eq:trace-lower}
 \HH\bigl(L_i\cap(\widehat E_0\cup\widehat E_1)\bigr)
 \ge\frac{2h}{\mu_0+\tan\theta}.
\end{equation}
This segment can be guaranteed to lie in $4\rho D$.

\medskip
\noindent\emph{Step 4: compare packet area with endpoint trace.}
Points of $L_i\cap\widetilde T$ whose abscissa is not in $X_i$ lie in the $h$-neighborhood of $\partial\widetilde T$.  Steiner's bound yields
\[
 |N_h(\partial\widetilde T)|
 \le2h\,\per(\widetilde T)+\pi h^2
 \le12.7\rho w.
\]
Because $|T|=2\rho^2\sin\beta$,
\[
 220|N_h(\partial\widetilde T)|
 \le2800\rho w
 \le\frac12\eta_0|T|.
\]
Because the ribbons cover the planks, the ribbon multiplicity bound and vertical height $h$ imply
\[
 \eta_0|T|
 \le \Big|\widetilde T\cap\bigcup \cR_0 \Big|
 \le \sum_i |L_i\cap\widetilde T|
 \le 220|N_h(\partial\widetilde T)| + \sum_{i:X_i\ne\varnothing}2h|X_i|.      
\]
The preceding boundary estimate therefore implies
\[
 \sum_{i:X_i\ne\varnothing}2h|X_i|
 \ge\frac12\eta_0|T|
 =\eta_0\rho^2\sin\beta.
\]
Because $|X_i|\le5.8\rho$,
\[
 \sum_i\HH\bigl(L_i\cap(\widehat E_0\cup\widehat E_1)\bigr)
 \ge\frac{\eta_0\rho\sin\beta}
 {5.8(\mu_0+\tan\theta)}.
\]
Ribbon multiplicity is at most $220$, so the union of all ribbon traces on the two arms has length at least
\[
 \frac{\eta_0\rho\sin\beta}
 {1276(\mu_0+\tan\theta)}.
\]
One of the two arms carries at least half of this length. Call the corresponding plank $S$.  By \eqref{eq:ribbon-dense}, every trace point $z$ on that arm has $B(z,w/2)\subset S\cap U_{\delta_1}$.  Since the centers are collinear, the union of these disks has area at least $(w/2)$ times the trace length.  Therefore,
\[
 |U_{\delta_1}\cap S\cap5\rho D|
 \ge\frac{\eta_0\sin\beta}
 {5200(\mu_0+\tan\theta)}\rho w.
\]
Finally, the disks lie in $5\rho D$ because their centers lie in $4\rho D$ and $w/2<\rho$.  
\end{proof}

\subsection{Proof of Proposition \ref{prop:pair-payment}}

Let $\eta=\tau/40=2.5\cdot10^{-6}$ and $\beta_0=0.06$.  Lemma~\ref{lem:triangle-reduction} gives $\cR$, a distinguished arm, and $T$.

Suppose first that $\beta\le\beta_0$.  The directions in $\cR$ lie in an interval of length $4\beta$.  Divide it into $121$ equal intervals and retain a subfamily $\cR_0$ covering at least $\eta|T|/121$.  Since $4\beta/121\le(1/30)\sin\beta$, every selected direction is within $\theta=(1/60)\sin\beta$ of the midpoint $d$ of its interval.  Take
\[
 \eta_0=\frac\eta{121},\qquad
 \theta=\frac1{60}\sin\beta,\qquad
 \psi_0=3\beta.
\]
Then \eqref{eq:theta-cond} holds and
\[
 \mu_0+\tan\theta \le\tan(3\beta)+\tan(\beta/60)\le3.1\beta.
\]
The width requirement in \eqref{eq:packet-cond} follows from \eqref{eq:thin-A}.
Lemma \ref{lem:main-packet} therefore gives
\[
 |U_\delta\cap(P\cup Q)\cap5\rho D|
 \ge \frac{\eta}{121\cdot5200\cdot3.1}\frac{\sin\beta}{\beta}\rho w
 >\gamma \rho w.
\]

Now suppose $\beta\ge\beta_0$.  Partition all unoriented directions $[0,\pi)$ into $1600$ equal intervals and select a subfamily covering at least $\eta|T|/1600$.  Set
\[
 \eta_0=\frac\eta{1600},\qquad
 \theta=\frac\pi{3200},\qquad
 \psi_0=\frac\pi2.
\]
Then \eqref{eq:theta-cond} holds, and
$\mu_0+\tan\theta<1.01$. 
The width condition again follows from \eqref{eq:thin-A} Hence
\begin{equation*}
 |U_\delta\cap(P\cup Q)\cap5\rho D|
 \ge\frac{\eta\sin0.06}{1600\cdot5200\cdot1.01}\rho w
 > \gamma \rho w. 
\end{equation*}
\qed 

\section{The cost of a hole}

We now combine the pruning procedure and the flattening mechanism to prove the main theorem.
By subdividing the planks if needed, we can assume that they have a common width that is as small as we wish.

\begin{proof}[Proof of Theorem~\ref{thm:hole-cost}]
Run the pruning procedure.  Let
$(P_1,Q_1),\dots,(P_m,Q_m)$ be the removed pairs and let $H'$ be the final hole.  
By Theorem \ref{thm:pruning}, with $\lambda=\rho/(200r)\ge2$,
\[
 (\lambda-1)|H|\le|H'|-|H|
 \le\sum_{i=1}^m\bigl(|P_i\cap D|+|Q_i\cap D|\bigr)
 \le4wm.
\]
Thus
\begin{equation}\label{eq:number-deletions}
 m\ge\frac{\lambda|H|}{8w}=\frac{\rho|H|}{1600rw}.
\end{equation}
For every $i$, let
\[
 S_i=U_\delta\cap(P_i\cup Q_i)\cap5\rho D.
\]
At its deletion time the pair is removable.  Proposition \ref{prop:pair-payment} gives
\begin{equation}\label{eq:Si-payment}
 |S_i|\ge \gamma\rho w.
\end{equation}
On the set $S=\bigcup_iS_i$,
\[
 \sum_{i=1}^m\1_{S_i}\le\sum_{i=1}^m(\1_{P_i}+\1_{Q_i})\le M_{\cP}.
\]
Hence,
\begin{equation}\label{eq:S-sum}
 \sum_i|S_i|\le |U_\delta|+I(\cP).
\end{equation}
By the Vitali covering lemma, a disjoint subfamily of $\delta$-dense disks $B_j$ has fivefold enlargements covering $U_\delta$.  Therefore,
\[
 |U_\delta|\le25\sum_j|B_j|
 <\frac{25}{\delta}\sum_j\int_{B_j}(M_{\cP}(x)-1)_+ \; dx
 \le\frac{25}{\delta}I(\cP).
\]
We can conclude
\[
 \frac{\gamma\rho^2}{1600r}|H|
 \le m\gamma\rho w
 \le\sum_i|S_i|
 \le\frac{26}{\delta}I(\cP).
\]
Thus,
\[
 I(\cP)\ge\frac{\delta \gamma\rho^2}{41600r}|H|.
\qedhere
\]
\end{proof}

\section{Two reductions}
\subsection{From overlap to width}

\label{sec:CostToB}

We finish with the relative-width measure on the disk.  This is the standard bridge from an overlap estimate to a total-width estimate that was suggested already by Tarski~\cite{tarski1932uwagi}; see also~\cite{gardner1988relative}.

\begin{lemma}[relative-width measure]\label{lem:width-measure}
Define a probability measure on $D$ by
\begin{equation}\label{eq:mu}
 d\mu(x)=\frac{dx}{2\pi\sqrt{1-\norm x^2}}.
\end{equation}
If a plank is reduced so that its normal-coordinate interval lies in $[-1,1]$, then
\[
 \mu(P)=\frac{w(P)}2.
\]
\end{lemma}

\begin{proof}[Proof of Theorem \ref{thm:annulus}]
If the hole $H$ is empty, Lemma \ref{lem:width-measure} immediately gives $W(\cP):=\sum_{P \in \cP} w(P)\ge2$.  Otherwise, Theorem \ref{thm:hole-cost} and Lemma \ref{lem:width-measure} give
\[
 W(\cP)\ge2+\frac{|H|}{\pi}
 \left(\frac{c_0}{r}-\frac1{\sqrt{1-r^2}}\right).
\]
For small $r$, we get $W(\cP)>2$.
\end{proof}

\subsection{From plane to higher dimensions}

\label{sec:high-dim}
\begin{proof}[Proof of Theorem~\ref{thm:hole-cost-Rn}]
We first note that Theorem~\ref{thm:hole-cost} holds, with different 
constants, for every centrally symmetric planar convex body $C$.
Indeed, after an affine change of coordinates we may assume
\[
D\subset C\subset \sqrt2 D.
\]
If the hole $G\subset sC$ and $s$ is sufficiently small, then
$G\subset \sqrt2 sD\subset D$. Applying Theorem~\ref{thm:hole-cost} in $D$ and
using that $C\setminus D$ is completely covered gives
\[
\sum_{P\in\mathcal P}|P\cap C|
\ge |C|+\frac{c_1'}{s}|G|.
\]

We now work in $\mathbb R^n$. Fix a two-dimensional generic subspace $V$,
and write $\mathbb R^n=V\oplus V^\perp$. 
First, fix $y\in V^\perp$, and set
\[
K_y=\{z\in V:y+z\in K\},\qquad
H_y=\{z\in V:y+z\in H\}.
\]
Suppose $H_y\neq\varnothing$ and set
\[
A:=K_{y/r}.
\]
Because $H\subset rK$, 
$
H_y\subset rA.
$
Because $K$ is origin symmetric, $K_{-y/r}=-A$. Put
\[
\alpha=\frac{1+r}{2},\qquad \beta=\frac{1-r}{2}.
\]
By convexity,
$
\alpha A-\beta A\subset K_y.
$
As $\alpha=r+\beta$, this says
$rA+C\subset K_y$
for the origin symmetric
$C:=\beta(A-A).$
Choose $a_0\in A$ and set $z_0=ra_0$. Then
$z_0+C\subset K_y$ and
\[rA\subset z_0+r(A-A)
=z_0+\frac{2r}{1-r}C.
\]
Thus,
\[
H_y-z_0 \subset \frac{2r}{1-r}C.
\]
Apply the planar estimate
to $C$. Because the rest of $K_y$ is fully covered, 
\[
\sum_{P\in\mathcal P}|P\cap K_y|_2
\ge
|K_y|_2+
\frac{c_1'(1-r)}{2r}|H_y|_2
\ge
|K_y|_2+\frac{c_1}r|H_y|_2.
\]
If $H_y=\varnothing$, the same inequality is immediate.

Finally, by integrating over $y\in V^\perp$ and Fubini's theorem,
\[
\sum_{P\in\mathcal P}|P\cap K|
\ge
|K|+\frac{c_1}r|H|.
\qedhere
\]
\end{proof}

\bibliographystyle{amsplain}
\bibliography{planks}

\end{document}